\documentclass{article}

\usepackage{amsmath,times} \usepackage{amsthm}
\usepackage{amssymb} \usepackage{latexsym} \usepackage{amscd}
\usepackage{verbatim} \usepackage[english]{babel}
\usepackage[latin1]{inputenc} \newtheorem{theorem}{Theorem}[section]
\newtheorem{lemma}[theorem]{Lemma}
\newtheorem{proposition}[theorem]{Proposition}
\newtheorem{corollary}[theorem]{Corollary}
\newcommand{\usc}{\operatorname{usc}}

\theoremstyle{definition} \newtheorem{definition}[theorem]{Definition}
\newtheorem{example}[theorem]{Example}
\newtheorem{remark}[theorem]{Remark}

\begin{document}

%\pubyr{1000} %%% Publishing Year
%%\vol{1} %%%% Volume number
%\iss{1} %%%% Issue number
%\startpage{1} %%% Starting pagenumber
%\endpage{1000} %%% End pagenumber

\title{Analytic test configurations and geodesic rays} 
\author{Julius Ross and David Witt Nystr\"om}
\date{5 June 2013}
%\thanks{\noindent {\sc Julius Ross,  University of Cambridge, UK. \\j.ross@dpmms.cam.ac.uk}\vspace{2mm}\\ 
%\noindent{\sc David Witt Nystr\"om,  University of Gothenburg, Sweden. \\\quad wittnyst@chalmers.se}
%
%}

\maketitle

\begin{abstract}
  Starting with the data of a curve of singularity types, we use the
  Legendre transform to construct weak geodesic rays in the space of
  locally bounded metrics on an ample line bundle $L$ over a compact
  manifold.  Using this we associate weak geodesics to suitable
  filtrations of the algebra of sections of $L$.  In particular this
  works for the natural filtration coming from an algebraic test
  configuration, and we show how this recovers the weak
  geodesic ray of Phong-Sturm.
\end{abstract}

\newcommand{\Conv}{\operatorname{Conv}}

% \tableofcontents

\section{Introduction}
Let $\mathcal H(L)$ be the space of smooth strictly positive hermitian
metrics on an ample holomorphic line bundle $L$ over a compact manifold $X$.
Then, by the work of Mabuchi \cite{Mabuchi}, Semmes \cite{Semmes} and Donaldson \cite{Donaldson3}, formally $\mathcal H(L)$ has the structure of an
infinite dimensional symmetric space with a canonical Riemannian
metric.  Thus a natural way to study this space is through its
geodesics, an approach that has been taken up by a number of authors
(e.g.\ Berndtsson \cite{Berndtsson}, Chen-Tian \cite{Chen2}, Donaldson, Phong-Sturm \cite{Sturm4,Sturm}, Mabuchi, and Semmes among others).

In this paper we give a general method for constructing weak geodesic rays
in the space of locally bounded positive metrics on $L$.  In the following we shall, in the standard way, identify a (positive) metric $h$ on $L$ with its (plurisubharmonic) potential $\phi=\log h$.   Our initial data consists of a fixed smooth plurisubharmonic potential $\phi$ and a curve of singular plurisubharmonic potentials $\psi_{\lambda}$ on $L$ for $\lambda\in \mathbb R$.   We are really  only interested in the singularity
type of $\psi_{\lambda}$, so we consider the equivalence class of
$\psi_{\lambda}$ under the relation $\psi_{\lambda}\sim
\psi'_{\lambda}$ if $\psi_{\lambda}-\psi'_{\lambda}$ is bounded
globally on $X$.  We define the \emph{maximal envelope} of this data
to be
\begin{equation}
  \label{eq:envelope}
 \phi_\lambda := {\sup}^* \{ \psi : \psi \le \phi \text{ and } \psi\sim \psi_{\lambda}\}  
\end{equation}
where the supremum is over positive metrics $\psi$ with the same
singularity type as $\psi_{\lambda}$, and the star denotes the
operation of taking the upper-semicontinuous regularization.

This equivalence relation on potentials was considered by Demailly-Peternell-Schn\-eider \cite{Demailly3} and is relevant to metrics with minimal singularities.   The envelope in \eqref{eq:envelope} was studied in a special case by Berman \cite{Berman4}, and is the global analogue of a construction of Rashkovskii-Sigurdsson \cite{Rashkovskii}.  Observe that if $\psi_{\lambda}$ is itself locally bounded then $\phi$ is a candidate for this envelope implying $\phi_{\lambda} = \phi$, and thus we are most interested in the case that $\psi_{\lambda}$ is singular along some non-trivial subset of $X$.   When $\psi_{\lambda}$ has analytic singularities, it can be shown that $\phi_{\lambda}$ is the largest plurisubharmonic potential bounded above by $\phi$ with the same singularity type as $\psi_{\lambda}$ (see Remark \ref{rmk:singularitytype}).   \medskip

We shall call a curve $\psi_{\lambda}$ of plurisubharmonic potentials a \emph{test curve} if it is concave in $\lambda$, locally bounded for $\lambda$ sufficiently negative and identically $-\infty$ for $\lambda$ sufficiently large (we also make one further technical condition concerning the kind of singularity allowed, see Definition \ref{definitiontestcurve}).  A justification of this terminology will be given below, and as a simple example to have in mind, suppose that $s$ is a holomorphic section of $L$ scaled so $|s|^2_{\phi}\le 1$  and let
\begin{equation}\label{eq:exampletestcurve}
\psi_{\lambda} = \left\{
  \begin{array}{ll}
    \phi & \lambda< 0 \\
(1-\lambda)\phi + \lambda \ln |s|^2 & 0\le \lambda<1\\
-\infty & \lambda\ge 1,
  \end{array}
\right.  
\end{equation}
so $\psi_{\lambda}$ is decreasing and concave in $\lambda$.  Geometrically then $\phi_{\lambda} = \phi$ for $\lambda\le 0$ and for $\lambda\in (0,1)$ it is the largest plurisubharmonic potential bounded above by $\phi$ with Lelong number at least $\lambda$ along the divisor $Y= \{ s=0\}$.    

\begin{theorem}\label{theoremmain1}
  Let $\psi_{\lambda}$ be a test curve and  $\phi$ be a smooth positive potential,  and consider the Legendre transform of its maximal envelope $\phi_{\lambda}$ given by
$$ \widehat{\phi}_t := {\sup}_{\lambda}^* \{ \phi_\lambda + \lambda t\} \quad \text{ for }t\in [0,\infty).$$
Then $\widehat{\phi}_t$ is a weak geodesic ray in the space of locally
bounded plurisubharmonic potentials on $L$ that emanates from $\phi$.
\end{theorem}

Looking back at \eqref{eq:exampletestcurve} it is clear that specific examples of test curves are easy to produce, and thus this theorem gives a new construction of many weak geodesic rays emanating from $\phi$.  In fact, as we shall observe in Remark \ref{rmk:Berndtsson}, essentially every weak geodesic ray can be produced in this way.\medskip

We recall precisely what is meant by a weak geodesic.    Let $I\subset \mathbb R$ be an interval and consider the annulus $A:=\{w\in \mathbb C: -\ln |w|\in I\}$ and $\pi\colon X\times A\to X$ be the projection.    Then a curve of plurisubharmonic potentials $\phi_t$ for $t\in I$ can be identified with a rotation invariant potential $\Phi(x,w):=\phi_{-\log |w|}$ on $\pi^*L$.  A simple calculation \cite[2.3,2.4]{Semmes} reveals that if  $\phi_t$ is a smooth curve of smooth potentials then the geodesic equation for $\phi_t$ is equivalent
to the homogeneous complex Monge-Amp\`ere equation
\begin{equation}
  MA(\Phi) : = (dd^c \Phi)^{n+1} =0 \quad \text{ on } X\times A^{\circ}\label{geodesicMA}.
\end{equation}
A curve $(\phi_t)_{t\in I}$ of locally bounded (not necessarily smooth) plurisubharmonic potentials given by $\Phi = \Phi(x,w)$ is said to be a \emph{weak geodesic} if $dd^c\Phi$ is positive and solves \eqref{geodesicMA} in the sense of currents.   When $I=[0,\infty)$ (so $A$ is the punctured unit disc) we say that $\phi_t$ is a \emph{weak geodesic ray}.  \medskip

The first step in our approach to Theorem \ref{theoremmain1} is
showing that the Monge-Amp\`ere measure of the maximal envelope
$\phi_{\lambda}$ satisfies
\begin{equation} \label{maxcd}
MA(\phi_{\lambda}) = \mathbf{1}_{\{\phi_{\lambda}=\phi\}} MA(\phi_{\lambda})\quad \text{ for all } \lambda,
\end{equation}
where ${\mathbf 1}_S$ denotes the characteristic function of a set $S$. We say that a positive metric $\phi_{\lambda}$ bounded by $\phi$ and having property (\ref{maxcd}) is \emph{maximal} with respect to $\phi$ (see Definition \ref{defmax2}), and a test curve $\phi_{\lambda}$ where $\phi_{\lambda}$ is maximal with respect to $\phi$ for all $\lambda$ is referred to as a \emph{maximal test curve}.   We shall show that the Legendre transform $\widehat{\phi}_t$ of a maximal test curve $\phi_{\lambda}$ is a subgeodesic, and also that Aubin-Mabuchi energy is linear in $t$, which together imply that it is a weak geodesic   (in the smooth case this is well known and goes back to Mabuchi \cite{Mabuchi2} and more generally can be deduced from a Theorem of Berman-Boucksom-Guedj-Zeriahi, see Lemma \ref{lemmageo2}).\medskip

The well known Yau-Tian-Donaldson conjecture states that for a
smooth projective manifold it should be possible to detect the
existence of a constant scalar curvature K\"ahler metric algebraically.  Through ideas
developed by many authors (e.g.\ Chen, Donaldson, Mabuchi, Tian) a
general picture has emerged in which such metrics appear as critical
points of certain energy functionals that are convex along smooth
geodesics.  The input from algebraic geometry arises through
notion of an (algebraic) test configuration, originally due to Tian \cite{Tian97} and then extended by Donaldson \cite{Donaldson2}, which, roughly speaking, is
a one-parameter algebraic degeneration of the original projective
manifold $X$.

In a series of papers, Phong-Sturm show how one can naturally associate
a weak geodesic ray to a test configuration \cite{Sturm4, Sturm,
  Sturm2} (see also \cite{Arezzo, Chen00, Chen08, Sun, Zelditch, Tian} for other constructions of geodesic rays related to test configurations).    We show how the geodesic of Phong-Sturm can be viewed as a particular case of the Legendre transform construction.  For example,  applying the Legendre transform construction to the example of the test curve \eqref{eq:exampletestcurve} given above recovers the geodesic of Phong-Sturm that is associated to the test configuration given by the degeneration to the normal cone of the divisor $Y$.\medskip

Generalizing slightly, suppose that $\mathcal F_{k,\lambda}$, for
$k\in \mathbb N, \lambda\in \mathbb R$ is a multiplicative filtration
of the graded algebra $\oplus_k H^0(X,kL)$.  Using the underlying smooth positive metric $\phi$ and some auxiliary volume form we have an $L^2$-inner product on each $H^0(X,kL)$, and thus can consider the associated Bergman metric
$$\phi_{k,\lambda} = \frac{1}{k} \ln \sum_{\alpha} |s_{\alpha}|^2$$
where $\{s_{\alpha}\}$ is an orthonormal basis for $\mathcal
F_{k,\lambda k}\subset H^0(X,kL)$.  So, for example, if $Y$ is a divisor on $X$ we can let $\mathcal F_{k,\lambda} = \{ s\in H^0(kL) : \text{ord}_Y(s)\ge \lambda\}$ be the multiplicative filtration given by order of vanishing along $Y$, and then $\phi_{k,\lambda}$ is the potential associated to the partial Bergman kernel coming from sections of $kL$ that vanish to at least order $k\lambda$ along $Y$.

\begin{theorem}\label{theoremmain2}
Suppose that $\mathcal F_{k,\lambda}$ is left continuous and decreasing in $\lambda$ and bounded (see \eqref{def}).  Then there is a well-defined limit
$$\phi_{\lambda}^{\mathcal F} = {\lim}^*_{k\to \infty}  \phi_{k,\lambda}$$
where the star denotes taking the upper semicontinuous regularization after the limit.   Furthermore this limit is maximal except possibly for one critical value of $\lambda$, and its Legendre transform is a weak geodesic ray.
\end{theorem}

In particular this applies to a natural filtration associated to a test configuration, and thus we have associated a weak geodesic to any such test configuration.  We prove that, in this case, we recover precisely Phong-Sturm geodesic, thus reproving the main result of \cite{Sturm}.  Hence one interpretation of Theorem \ref{theoremmain1} is that in the problem of finding weak geodesic rays,  the algebraic data of a test configuration can be replaced with a suitable concave curve of singularity types which we thus refer to as an \emph{analytic test configuration}. \medskip

The relationship between a algebraic and analytic test configurations is analogous to the relationship between holomorphic and plurisubharmonic functions. Being given by a concave curve of singularity types, analytic test configurations are easier to produce and manipulate.    For example it is possible to interpolate between two analytic test configurations merely by taking the line between them, whereas it is not so clear what the analogous construction is for algebraic test configurations.   With regard to the Yau-Tian-Donaldson conjecture it is now expected that the notion of algebraic test configuration needs to be extended in some way for it to detect a constant scalar curvature metric (see the examples of Apostolov-Calderbank-Gauduchon-Friedman \cite{Apostolov} and also work of Sz\'ekelyhidi \cite{Sz} which considers general filtrations of the ring of sections).  Thus analytic test configurations provide a context for such extensions.\medskip

Interesting examples of analytic test configurations that do not correspond to algebraic ones appear for example in the setting of Arakelov geometry. In \cite{Chen} Boucksom-Chen construct multiplicative filtrations of the section ring which encodes arithmetic properties of the sections such as their adelic norm. This filtration then gives rise to an analytic test configuration whose geometrical significance remains unclear. \medskip

Another advantage of the notion of analytical test configurations is that they can be defined independent of the polarization, and even in the non-projective case (although we will not consider that further here). If we pick an arbitrary Kähler form $\omega$ we can consider a concave curve of singularity types $[\psi_{\lambda}]$ of $\omega$-plurisubharmonic functions. Then choosing a polarization $L$ we get an analytic test configuration $[\phi_{\lambda}]$ by letting $\phi_{\lambda}$  be the envelope of all positive singular metrics of $L$ bounded from above by $\phi+\psi_{\lambda},$ where $\phi$ is some locally bounded metric of $L.$   In fact, rather than choosing a polarization we could have picked an arbitrary big $(1,1)$-coholomogy class and worked with that instead.\medskip

It should be stressed that in the problem of finding constant scalar curvature metrics it is important to have control of the regularity of geodesics under consideration.  By using approximations to known regularity results of solutions of Monge-Amp\`ere equations,  Phong-Sturm prove that their weak geodesic is in fact $C^{1,\alpha}$ for $0<\alpha<1$ (see \cite{Sturm2}). It is interesting to ask whether such  regularity holds more generally, which is a topic we take up in \cite{Ross3}.  \medskip

\noindent \textbf{Organization: } We start in Section
\ref{sectionconvex} with some motivation from convex analysis and toric geometry, and
Section \ref{sectionprelim} contains preliminary material on the space
of singular metrics, the Monge-Amp\`ere measure and the Aubin-Mabuchi
functional.  The real work starts in Section \ref{sectionenvelopes}
where we consider the maximal envelopes associated to a given
singularity type.  Along the way we prove a generalization of a
theorem of Bedford-Taylor which says that such envelopes are maximal (Theorem \ref{lemma3}).  This is then extended to the case
of a test curve of singularities, and in Section \ref{sectionlegendre}
we discuss the Legendre transform and prove Theorem
\ref{theoremmain1}.

Following these analytic results, we move on to the algebraic picture.  In Section \ref{filt} we associate a test curve to a suitable filtration of the coordinate ring of $(X,L)$, and prove Theorem \ref{theoremmain2}.   We then recall how such filtrations arise from test configurations, and in Section \ref{sectionphongsturm} show how this agrees with the construction of Phong and Sturm.\medskip

\noindent\textbf{Acknowledgments}: We would like to thank Robert
Berman, Bo Berndtsson, Sebastian Boucksom, Yanir Rubinstein and Richard Thomas for
helpful discussions and the referees for their constructive suggestions.     We also thank Dano Kim for pointing out a mistake in a previous version of this paper.  The first author also acknowledges support from a Marie Curie International Reintegration Grant within the 7\textsuperscript{th} European Community Framework Programme (PIRG-GA-2008-230920). 

\section{Convex motivation}\label{sectionconvex}

This section contains some motivation from convex analysis in the study of the homogeneous Monge-Amp\`ere equation.   Much of this material is standard; our main references are the two papers \cite{Yanir} and \cite{Yanir2} by Rubinstein-Zelditch.  We shall presently see how solutions to this equation can be found using the Legendre transform in two different, but ultimately equivalent, ways.   Although this is logically independent of the rest of the paper, the techniques used are similar and give an illustration in the toric setting.\medskip

Let $\Conv(\mathbb{R}^n)$ denote the space of convex functions on
$\mathbb{R}^n.$ We take the
convention that the function identically equal to $-\infty$ is in
$\Conv(\mathbb{R}^n).$

\begin{definition}
  Let $\phi$ be a $C^2$ convex function on an open subset of
  $\mathbb{R}^n$. The \emph{(real) Monge-Amp\`ere measure} of
  $\phi,$ denoted by $MA(\phi),$ is the Borel measure defined as
  $$MA(\phi):=d\frac{\partial \phi}{\partial x_1}\wedge ... \wedge
  d\frac{\partial \phi}{\partial x_{n+1}}.$$  The operator $MA$ has an unique extension to a continuous operator on the cone of (finite-valued) convex functions (see \cite{Yanir2} for references). If $\phi$ is $C^2$
  then
  \begin{equation} \label{monge} MA(\phi)=det(\nabla ^2\phi)dx=(\nabla
    \phi)^*dx,
  \end{equation} 
  i.e.\ the Monge-Amp\`ere measure is the pullback of the Lebesgue
  measure under the gradient map.
\end{definition}

If $\phi\in \Conv(\mathbb{R}^n)$, we say $y$ is a \emph{subgradient} of $\phi$ at $x_0$ if the function $x\mapsto \phi(x) - \phi(x_0) - y\cdot(x-x_0)$ is bounded from below, and we let\[\Delta_{\phi}=\{y: y \text{ is a subgradient for some } x_0\}\] be the set of all subgradients of $\phi$.   So, if $\phi$ is differentiable then
$\Delta_{\phi}$ is simply the image of $\nabla\phi$. One can easily
check that $\Delta_{\phi}$ is convex, that if $r>0$ then
$\Delta_{r\phi} = \Delta_{\phi}$ and $\Delta_{\phi+\psi}\subseteq \Delta_{\phi}+\Delta_{\psi}.$

When $\phi$ is $C^2$ it follows from equation (\ref{monge}) that the
total mass of the Monge-Amp\`ere measure $MA(\phi)$ equals the Lebesgue
volume of $\Delta_{\phi}$. An important fact \cite{Yanir2} is
that this is true for all convex functions on ${\mathbb R}^n$ with linear growth, i.e.
\begin{equation} \label{monge2}
  \int_{\mathbb{R}^n}MA(\phi)=vol(\Delta_{\phi}).
\end{equation}

We say two convex functions $\phi$ and $\psi$ are
\emph{equivalent} if $|\phi-\psi|$ is bounded, and denote this by $\phi\sim \psi$. Since for
two equivalent convex functions $\phi$ and $\psi$ with linear growth we clearly have that
$$\Delta_{\phi}=\Delta_{\psi},$$ it follows from 
\eqref{monge2} that
$$\int_{\mathbb{R}^n}MA(\phi)=\int_{\mathbb{R}^n}MA(\psi)\quad\text{
  whenever } \phi \sim \psi.$$

\begin{definition}
  Let $\phi\in \Conv(\mathbb{R}^n)$ and let $\dot \phi$ be a bounded
  continuous function on $\mathbb{R}^n$. A curve $\phi_t$ in
  $\Conv(\mathbb{R}^n)$, $t\in [a,b],$ is said to solve the
  \emph{Cauchy problem} for the homogeneous real Monge-Amp\`ere
  equation, abbreviated as HRMA, with initial data $(\phi,\dot \phi),$ if the function $\Phi(x,t):=\phi_t(x)$  is convex on $\mathbb{R}^n\times [a,b]$, and satisfies the equation
$$ MA(\Phi)=0\quad \text{ on the strip }\quad \mathbb{R}^n\times (a,b),$$
with initial data
 $$\phi_0=\phi, \qquad \frac{\partial}{\partial
    t}_{|t=0^+}\phi_t=\dot \phi.$$ 
\end{definition}

\begin{remark}
  This convex geometry has particular geometric significance when $\Delta$ is the moment polytope of a polarised toric manifold $(X,L)$.  Then, through the notion of symplectic potentials, there is a correspondence between hermitian metrics on $L$ and convex functions on $\Delta$ (where due care is to be taken to ensure that a given convex function determines a smooth or positive metric on $L$) and the Cauchy problem for the HRMA translates to the Cauchy problem for finding geodesics in the space of hermitian metrics on $L$.  We refer the reader to \cite{Yanir,Yanir2} for a detailed discussion of this idea.
\end{remark}

Now let $\phi_0$ and $\phi_1$ be two equivalent convex functions with linear growth, and $\phi_t$ be the affine
curve between them. The \emph{energy} of $\phi_1$ relative to
$\phi_0,$ denoted by $\mathcal{E}(\phi_1,\phi_0)$ is defined as
$$\mathcal{E}(\phi_1,\phi_0):=\int_{t=0}^1\left(\int_{\mathbb{R}^n}(\phi_1-\phi_0)MA(\phi_t)\right)dt.$$
We observe that by the linear growth assumption it follows that the
relative energy $\mathcal{E}(\phi_1,\phi_0)$ is finite.  This energy
has a cocycle property, namely if $\phi_0,$ $\phi_1$ and $\phi_2$ are
equivalent with finite energy then
$$\mathcal{E}(\phi_2,\phi_0)=\mathcal{E}(\phi_2,\phi_1)+\mathcal{E}(\phi_1,\phi_0),$$
which is easily seen to be equivalent to the fact that
$$\frac{\partial}{\partial
  t}\mathcal{E}(\phi_t,\phi)=\int_{\mathbb{R}^n}\frac{\partial}{\partial
  t}\phi_t MA(\phi_t).$$ The energy along a smooth curve $\phi_t$ of convex functions with linear growth is related
to the Monge-Amp\`ere measure of $\Phi(x,t):=\phi_t(x)$ by the identity
\begin{equation} \label{volumeformula} \int_{\mathbb{R}^n\times
    [a,b]}MA(\Phi)=\frac{\partial}{\partial
    t}_{|t=b}\mathcal{E}(\phi_t,\phi_a)-\frac{\partial}{\partial
    t}_{|t=a}\mathcal{E}(\phi_t,\phi_a).
\end{equation}
Thus a smooth curve $\phi_t$ of equivalent convex functions of linear growth solves the HRMA equation if and only if $\Phi$ is convex and the energy $\mathcal{E}(\phi_t,\phi_a)$ is linear in $t$.\medskip

As is noted in \cite{Yanir2} this Cauchy problem is not always solvable.
Nevertheless there is a standard way to produce solutions $\phi_t$ with $t\in [0,\infty)$ to the homogeneous Monge-Amp\`ere equation with given starting point
$\phi_0=\phi$ using the Legendre transform. We
give a brief account of this.

For simplicity assume from now on that $\phi$ is differentiable and
strictly convex.  Recall that the Legendre transform of $\phi$ is the function on $\Delta_{\phi}$ defined as
$$\phi^*(y):=\sup_x\{x\cdot y-\phi(x)\}$$ (which we can also think of
as being defined on the whole of $\mathbb{R}^n$, by being $-\infty$
outside of $\Delta_{\phi}$). Since $\phi^*$ is defined as the supremum
of the linear functions $x\cdot y-\phi(x),$ it is convex. In fact, one
can show that $\phi$ being differentiable and strictly convex essentially implies that $\phi^*$ is also differentiable and strictly convex (see \cite[Theorem 1]{Rockafellar} for a precise statement that requires a further boundary condition).

For a given $y\in \Delta_{\phi},$ the function $x\cdot y-\phi(x)$ is
strictly concave, and is maximized at the point where the gradient is
zero. Thus 
\begin{equation} \label{eq111} \phi^*(y)=x\cdot y-\phi(x) \quad \text{
    where }\quad \nabla\phi(x) =y,
\end{equation}
and hence
$$\nabla \phi^*(y)=x \quad \text{ where }\quad  \nabla\phi(x) =y.$$
The Legendre transform is an involution.  For using the above formula $\nabla \phi^{**}(x)=y$  for $x$ such
that $\nabla \phi^*(y)=x$ which holds when $\nabla \phi(x)=y,$ i.e.
$$\nabla \phi^{**}(x)=\nabla \phi(x).$$ If $\nabla \phi(x)=y$, then $\phi^*(y)=x\cdot y-\phi(x),$ therefore
$$\phi^{**}(x)=x\cdot y-\phi^*(y)=x\cdot y-(x\cdot y-\phi(x))=\phi(x),$$ as claimed.

\begin{lemma} \label{lemma21} If $\phi_t$ is a curve of convex
  functions, then for any point $y\in \Delta_{\phi_t}$
  $$\frac{\partial}{\partial t} \phi_t^*(y)=-\frac{\partial}{\partial
    t}\phi_t(x),$$ where $x$ is the point such that $\nabla\phi(x)=y$.
\end{lemma}

\begin{proof}
Let $x_t$ be the solution to the equation $\nabla\phi_t(x_t)=y$. By
  the implicit function theorem $x_t$ varies smoothly with $t$. By
  equation (\ref{eq111}) we know $$\frac{\partial}{\partial
    t} \phi_t^*(y)=\frac{\partial}{\partial t}(x_t\cdot
  y-\phi_t(x))=\frac{\partial}{\partial t}(x_t\cdot
  y-\phi(x))-\frac{\partial}{\partial t}\phi_t(x).$$ Since $x_t\cdot
  y-\phi(x)$ is maximized at $x=x_0$ the derivative of that part
  vanishes at $t=0,$ so we get the lemma for $t=0$, and similarly for
  all $t$.
\end{proof}

This leads us to the following formula relating the energy with the
Legendre transform,

\begin{lemma}
  We have
  \begin{equation} \label{eqen}
    \mathcal{E}(\phi_t,\phi)=\int_{\Delta_{\phi}}(\phi^*-\phi_t^*)dy.
  \end{equation}
\end{lemma}

\begin{proof}
  We noted above that the derivative with respect to $t$ of the
  left-hand side of (\ref{eqen}) is equal to
  $$\int_{\mathbb{R}^n}\frac{\partial}{\partial t}\phi_t MA(\phi_t).$$
  On the other hand, differentiating the right-hand side yields
  \begin{eqnarray*}
    \frac{\partial}{\partial t}\int_{\Delta_{\phi}}(\phi^*-\phi_t^*)dy=-\int_{\Delta_{\phi}}\frac{\partial}{\partial t} \phi_t^*dy=\int_{\Delta_{\phi}}\frac{\partial}{\partial t}\phi_t(\nabla\phi_t^{-1}(y))dy=\\ \int_{\mathbb{R}^n}\frac{\partial}{\partial t}\phi_t (\nabla \phi_t)^*dy=\int_{\mathbb{R}^n}\frac{\partial}{\partial t}\phi_t MA(\phi_t),
  \end{eqnarray*}
  where we used Lemma \ref{lemma21} and the fact that the pullback $(\nabla
  \phi_t)^*dy$ of the Lebesgue measure is $MA(\phi_t)$. Since both sides of the equation
  (\ref{eqen}) is zero when $\phi_t=\phi$ and the derivatives
  coincide, we get that they must be equal for all $t$.
\end{proof}

Now fix a smooth bounded strictly concave function $u$ on $\Delta_{\phi}$ and let 
\begin{equation}
\tilde{\phi}_t:= (\phi^* - tu)^*.\label{eq:defphitilde}  
\end{equation}

\begin{proposition} \label{prophrma} The curve $\tilde{\phi}_t,$ $t\in
  [0,\infty)$ solves the HRMA equation with initial data
$$ \tilde{\phi}_0 = \phi \quad \text{ and } \quad \frac{\partial}{\partial t}|_{t=0^+} \tilde{\phi}_t = u((\nabla \phi)^{-1}).$$
\end{proposition}

To see this note that from (\ref{eqen}) it follows that $$\mathcal{E}(\tilde{\phi}_t,\phi)=\int_{\Delta_{\phi}}(\phi^*-\tilde{\phi}_t^*)dy=\int_{\Delta_{\phi}}(\phi^*-\phi^*+tu)dy=t\int_{\Delta_{\phi}}udy,$$ which is linear in $t$, and the initial conditions follow from \eqref{eq:defphitilde} and Lemma \ref{lemma21}.    The convexity of $\tilde{\Phi}(t,x)= \tilde{\phi}_t(x)$ can of course be shown directly, but it also follows from another characterization of $\tilde{\phi}_t$ that also involves a Legendre transform, but in the $t$-coordinate instead of in the $x$-coordinates which we now discuss.\medskip

Let $A_{\lambda}$ be the subset of $\Delta_{\phi}$ where $u$ is
greater than or equal to $\lambda$ and let $\phi_{\lambda}$ be defined
as 
\begin{equation}
  \label{eq:defphilambda}
\phi_{\lambda}:=\sup\{\psi\leq \phi: \psi\in \Conv(\mathbb{R}^n),
\Delta_{\psi}\subseteq A_{\lambda}\}.
\end{equation}
\begin{lemma} \label{lemma34} The curve of functions $\phi_{\lambda}$
  is concave in $\lambda$ and $$\{\phi_{\lambda}=\phi\}=\{x: \nabla
  \phi(x)\in A_{\lambda}\}.$$
\end{lemma}

\begin{proof}
  Let $\psi_i\leq \phi$ be such that $\Delta_{\psi_i}\subseteq
  A_{\lambda_i}$ with $i=1,2$. Let $0<t<1$. From our discussion above
  it follows that $t\psi_1+(1-t)\psi_2\leq \phi$ and
  $$\Delta_{t\psi_1+(1-t)\psi_2}\subseteq
  t\Delta_{\psi_1}+(1-t)\Delta_{\psi_2}\subseteq
  tA_{\lambda_1}+(1-t)A_{\lambda_2}\subseteq
  A_{t\lambda_1+(1-t)\lambda_2},$$ where the last inclusion follows
  from the fact that $u$ was assumed to be concave.  For the second
  statement, it is easy to see that in fact $\phi_{\lambda}$ is equal
  to the supremum of affine functions $x\cdot y+C$ bounded by $\phi$
  and $y$ lying in $A_{\lambda}$.
\end{proof}

\begin{definition}
For $t\geq 0$ let $\widehat{\phi}_t$ be defined as
  $$\widehat{\phi}_t:=\sup_{\lambda}\{\phi_{\lambda}+t\lambda\}.$$
\end{definition}

Since for each $\lambda$ the function $(x,t)\mapsto
\phi_{\lambda}(x)+t\lambda$ is convex in all its variables, and the
supremum of convex functions is convex,
$\widehat{\Phi}(x,t):=\widehat{\phi}_t(x)$ is convex.

\begin{proposition} \label{prop01} Recalling that $\tilde{\phi}_t = (\phi^*-tu)^*$ we have $\tilde{\phi}_t=\widehat{\phi}_t$. In particular this proves that
  $\tilde{\Phi}$ is convex, thereby proving $\tilde{\phi}_t$ solves
  the HRMA equation (Proposition \ref{prophrma}).
\end{proposition}

\begin{proof}
  We claim
  \begin{equation} \label{eq700} \frac{\partial}{\partial
      t}\widehat{\phi}_t(x)=u(\nabla \widehat{\phi}_t(x)).
  \end{equation}
  To see this first consider the right-derivative at $t=0$. As we noted above, the gradient of a Legendre transform is the point where the maximum is attained, thus in this case $$\frac{\partial}{\partial t}|_{t=0^+}\widehat{\phi}_t(x)=\sup\{\lambda:
  \phi_{\lambda}(x)=\phi(x)\}.$$ By the second statement in Lemma
  \ref{lemma34} it follows that this supremum is equal to $u(\nabla
  \phi(x)),$ and we are done for $t=0$. On the other hand it is easy
  to see that $$\widehat{\phi}_{t_1+t_2}=\widehat{\psi}_{t_2},$$ with
  $\psi:=\widehat{\phi}_{t_1}$. Using this, the equation
  (\ref{eq700}) holds for all $t$. Thus by Lemma \ref{lemma21} the Legendre transform of $\widehat{\phi}_t$ is equal to $\phi-tu,$ so by the involution property of the Legendre transform $\widehat{\phi}_t$ coincides with $\tilde{\phi}_t$.
\end{proof}

Now the above discussion can be reformulated as follows.  Let
$\psi_{\lambda}$ be a curve in $\Conv(\mathbb{R}^n)$.  We say that $\psi_{\lambda}$ is a \emph{test curve}  if there is a $C$ such that
\begin{enumerate}
\item $\psi_{\lambda}$ is concave in $\lambda$
\item $\psi_{\lambda}-\phi$ bounded for $\lambda<-C$ and
\item $\psi_{\lambda}\equiv -\infty$ for $\lambda>C$.
\end{enumerate}
Given such a test curve let $u$ be
the function on $\Delta_{\phi}$ defined by $$u(y):=\sup\{\lambda: y\in
\Delta_{\psi_{\lambda}}\}.$$ 
and observe that since $\psi_{\lambda}$ is assumed to be concave the function $u$ is also concave.\medskip

Thus our definition of $\phi_{\lambda}$ above \eqref{eq:defphilambda} is
\begin{eqnarray}\label{eq:rephrase}
\phi_{\lambda}&=&
\sup\{\psi\leq \phi: \psi\in \Conv(\mathbb{R}^n),
\Delta_{\psi}\subseteq A_{\lambda}\}\nonumber\\
&=&\sup\{\psi\leq \phi: \psi\in \Conv(\mathbb{R}^n),
\Delta_{\psi}\subseteq \{u\geq \lambda\}\}.
\end{eqnarray}
Hence we in fact have
\begin{equation}
  \label{eq:toricenvelope}
\phi_{\lambda}:=\sup\{\psi:\psi\leq \phi, \psi\leq
\psi_{\lambda}+o(1), \psi\in \Conv(\mathbb{R}^n)\}.
\end{equation}

From Proposition \ref{prop01}, $\widehat{\phi}_t$ solves
the homogeneous real Monge-Amp\`ere equation. Thus in order to get
solutions to the HRMA, instead of starting with a concave function $u$
on $\Delta_{\phi}$ we can just as well start with a test curve
$\psi_{\lambda}$.   It is this second reformulation \eqref{eq:toricenvelope} that extends more naturally to the context of positive metrics on line bundles.\medskip

\begin{remark}
  This convex picture can be given some geometric context by considering toric manifolds (compare \cite[Example 5.2]{Berman}).   Consider a complex toric manifold $(X,L)$ of dimension $n$, so  $L$ carries a compatible torus invariant hermitian metric $\phi_0$.  Pick complex coordinates $z_1,\ldots,z_n$ on a dense complex torus $T=(\mathbb C^*)^n\subset X$ and let $t_i = \ln |z_i|^2$ so we consider $t_1,\ldots,t_n$  as coordinates on $\mathbb R^n$.  Then $\phi_0(z)$ descends to a convex function on $\mathbb R^n$ which, by abuse of notation, we denote by $\phi_0(t)$.  Conversely any function $\phi$ on $\mathbb R^n$ induces a metric on $L|_T$, which will be (strictly) positive as long as $\phi$ is (strictly) convex (one must make some additional hypothesis as to the behaviour of $\phi(t)$ at infinity if one wishes to ensure that this induced metric extends to a smooth or locally bounded metric over all of $X$, but we will not consider that further here).

As an illustration, consider the simplest case of dimension 1 so $X=\mathbb P^1$.    Fix a smooth strictly positive metric on $L$ and assume the strictly convex potential $\phi(t)$ on $\mathbb R$ satisfies $\phi'(0)=0$ and $\lim_{t\to \infty} \phi'(t) = 1$.   Then, for $\lambda\in (0,1)$ define
$$\psi_{\lambda}(t) =\left\{  \begin{array}{ll}
\lambda t & t\ge 0\\
0 & t\le 0
\end{array}\right.
$$
(this can be made into a test curve by extending it to be $\phi+C'$ for negative $\lambda$ and to be $-\infty$ for $\lambda\ge 1$ similar to the example \eqref{eq:exampletestcurve} in the introduction).    Now for $\lambda\in (0,1)$ let $t_{\lambda}$ be the point where $\phi'(t_{\lambda}) = \lambda$.  Then one sees from \eqref{eq:rephrase} that $\phi_{\lambda} = \phi$ for $t\le t_{\lambda}$ and $\phi_{\lambda} = \lambda (t-t_\lambda) + \phi(t_{\lambda})$ for $t\ge t_{\lambda}$.    One can check that $\phi_{\lambda}$ defines a singular metric on $L$ with the same singularity type as $\psi_{\lambda}$.     Observe that $\phi_{\lambda}$ is linear for $t>t_{\lambda}$ making this metric pluriharmonic (and thus maximal) over this set, and it is this maximal property that will be crucial in the generalisation that follows.
\end{remark}

\section{Preliminary Material}\label{sectionprelim}
We collect here some preliminary material on the space of positive
metrics, the (non pluripolar) Monge-Amp\`ere measure and the
Aubin-Mabuchi energy functional.  Most of this material is standard,
and we give proofs only for those results for which we did not find a
convenient reference.

\subsection{The space of positive singular metrics}

\newcommand{\PSH}{PSH} Let $X$ be a complex projective manifold of
complex dimension $n,$ and let $L$ be an ample line bundle on $X$. We start with some preliminaries on singular metrics, for which a convenient reference is \cite{Demaillybook}.  A
continuous (or smooth) hermitian metric on $L$ is a continuous (or smooth) choice of
scalar product on the complex line $L_p$ at each point $p$ on the
manifold. If $f$ is a local holomorphic frame for $L$ on $U_f$, then
one writes $$|f|_h^2=h_f=e^{-\phi_f},$$ where $\phi_f$ is a continuous (or smooth)
function on $U_f$. We will use the convention to let $\phi$ denote the
metric $h=e^{-\phi},$ thus if $\phi$ is a metric on $L,$ $k\phi$ is a
metric on $kL:=L^{\otimes k}$.

The curvature of a smooth metric is given by
$dd^c\phi$ which is the $(1,1)$-form locally defined as $dd^c\phi_f,$
where $f$ is any local holomorphic frame. Here $d^c$ is short-hand for
the differential operator $$\frac{i}{2\pi}(\bar{\partial}-\partial),$$
so $dd^c=i/\pi \partial \bar{\partial}$. A classic fact is that the
curvature form of a smooth metric $\phi$ is a
representative for the first Chern class of $c_1(L)$.  The metric $\phi$ is said to be strictly positive if the curvature $dd^c\phi$
is strictly positive as a $(1,1)$-form, i.e.\ if for any local holomorphic frame
$f,$ the function $\phi_f$ is strictly plurisubharmonic. We let
$\mathcal{H}(L)$ denote the space of smooth strictly positive (i.e.\
locally strictly plurisubharmonic) metrics on $L,$ which is non-empty
since we assumed that $L$ was ample.

A positive singular metric $\psi$ is a metric that can be written as
$\psi:=\phi+u,$ where $\phi$ is a smooth metric and $u$ is a globally defined $dd^c\phi$-psh function, i.e.\ $u$ is upper semicontinuous and
$dd^c\psi:=dd^c \phi+dd^c u$ is a positive $(1,1)$-current. For
convenience we also allow $u\equiv -\infty$. We let $\PSH(L)$ denote
the space of positive singular metrics on $L$.

As an important example, if $\{s_i\}$ is a finite collection of
holomorphic sections of $kL,$ we get a positive metric
$\psi:=\frac{1}{k}\ln(\sum |s_i|^2)$ which is defined by letting for
any local frame $f$, $$e^{-\psi_f}:=\frac{|f|^2}{(\sum
  |s_i|^2)^{1/k}}.$$

We note that $\PSH(L)$ is a convex set, since any convex combination
of positive metrics yields a positive metric. Another important fact
is that if $\psi_i \in \PSH(L)$ for all $i\in I$ are uniformly bounded above by some fixed positive metric, then the upper semicontinuous regularization of the supremum denoted by ${\sup}^*\{\psi_i : i\in I\}$ lies in $\PSH(L)$ as well.  Similarly the upper semicontinuous regularization of the pointwise limit $\lim_i \psi_i$ (when defined) will be denote by ${\lim}_i^* \psi_i$.

If $\psi$ is in $\PSH(L),$ then the translate $\psi+c$ where $c$ is a real constant
is also in $\PSH(L)$.   For any $\psi\in
\PSH(L),$ $dd^c\psi$ is a closed positive $(1,1)$-current, and from
the $dd^c$ lemma it follows that any closed positive current
cohomologous with $dd^c\psi$ can be written as $dd^c\phi$ for some
$\phi$ in $\PSH(L)$. By the maximum principle this $\phi$ is uniquely
determined up to translation.

If there exists a constant $C$ such that $\psi \leq \phi+C,$ we say
that $\psi$ is more singular than $\phi,$ and we will write this as
$$\psi \preceq \phi.$$ If both $\psi \succeq \phi$ and $\phi \succeq
\psi$ we say that $\psi$ and $\phi$ are \emph{equivalent}, which we write as
$\psi \sim \phi$. Following \cite{Guedj} an equivalence class $[\psi]$ is called a \emph{singularity type}, and we use the notation $Sing(L)$ for the set of singularity types. If $\psi$ is equivalent to an element in
$\mathcal{H}(L)$ we say that $\psi$ is \emph{locally bounded}.

The \emph{singularity locus} of a positive metric $\psi$ is the set where $\psi$ is minus infinity, i.e. the set where $\psi_f=-\infty$ when $f$ is a local frame and the \emph{unbounded locus} of $\psi$ is the set where $\psi$ is not locally bounded. Recall that a set is said to be \emph{complete pluripolar} if it is locally the singularity locus of a plurisubharmonic function, while it is called pluripolar if you only have a local inclusion in the singularity set. Pluripolar sets have zero measure with respect to any smooth volume form (since this is true locally with respect to the Lebesgue measure \cite[Corollary 2.9.10]{Klimek}).   In \cite{Guedj} Boucksom-Eyssidieux-Guedj-Zeriahi give the following definition. 

\begin{definition}
A positive metric $\psi$ is said to have \emph{small unbounded locus} if its unbounded locus is contained in a closed complete pluripolar subset of $X$.
\end{definition}

We note that metrics of the form $\frac{1}{k}\ln(\sum |s_i|^2)$ have small unbounded locus, since they are locally bounded away from the algebraic set $\cup_i\{s_i=0\}$ which is a closed pluripolar set.

\subsection{Regularization of positive singular metrics} If $f$ is a plurisubharmonic function on an open subset $U$ of $\mathbb{C}^n$ then using a convolution we can write $f$ as the limit of a decreasing sequence of smooth plurisubharmonic functions on any relatively compact subset of $U$ \cite[Theorem 2.9.2]{Klimek}.

If $\psi$ is a positive singular metric, we can use a partition of unity with respect to some open cover $U_{f_i}$ to patch together the smooth decreasing approximations of $\psi_{f_i}$. Thus any positive singular metric can be written as the pointwise limit of a decreasing sequence of smooth metrics, but of course because of the patching these smooth approximations will in general not be positive.

A fundamental result due to Demailly \cite{Demailly2} is that any positive singular metric can be approximated by metrics of the form $k^{-1}\ln(\sum_i |s_i|^2),$ where $s_i$ are sections of $kL$. Let $\mathcal{I}(\psi)$ denote the multiplier ideal sheaf of germs of holomorphic functions locally integrable against $e^{-\psi_f}dV,$ where $f$ is a local frame for $L$ and $dV$ is an arbitrary volume form. We get a scalar product $(.,.)_{k\psi}$ on the space $H^0(kL\otimes \mathcal{I}(k\psi))$ by letting $$\|s\|^2_{k\psi}:=\int_{X}|s|^2e^{-k\psi}dV.$$ Let $\{s_i\}$ be an orthonormal basis for $H^0(kL\otimes \mathcal{I}(k\psi))$  and set $$\psi_k:=\frac{1}{k}\ln(\sum |s_i|^2).$$

\begin{theorem} \label{demailly}
The sequence of metrics $\psi_k$ converge pointwise to $\psi$ as $k$ tends to infinity, and there exists a constant $C$ such that for large $k,$ $$\psi\leq \psi_k+\frac{C}{k}.$$
\end{theorem}

As a reference see \cite[(A4)]{Demailly} (note in fact the results of Demailly are in fact much stronger than that stated here, and hold in greater generality \cite{Demailly2}). When $\psi$ is assumed to be smooth and strictly positive, a celebrated result by Bouche-Catlin-Tian-Zelditch \cite{Bouche,Catlin,Tian,Zelditch2} on Bergman kernel asymptotics implies that the $\psi_k$ in fact converge to $\psi$ in any $C^m$ norm.  

Using a variation of this construction, Guedj-Zeriahi prove \cite[Theorem 7.1]{Guedj2} that any positive singular metric on an ample line bundle over a compact $X$ is the pointwise limit of a decreasing sequence of smooth positive metrics.

\subsection{Monge-Amp\`ere measures} \label{sectionenv} Let $\psi_i,$
$1\leq i\leq n,$ be an n-tuple of positive metrics, so for each $i,$
$dd^c\psi_i$ is a positive $(1,1)$-current. If all $\psi_i$ are smooth
one can consider the wedge product
\begin{equation} \label{product} dd^c\psi_1\wedge ...\wedge
  dd^c\psi_n,
\end{equation}
which is a positive measure on $X$. The fundamental work of
Bedford-Taylor shows that one can still take the wedge product of
positive currents $dd^c\psi_i$ to get a positive measure as
long as the metrics $\psi_i$ are all locally bounded. The Monge-Amp\`ere measure
of a locally bounded positive metric $\psi,$ is then defined as the
positive measure $$MA(\psi):=(dd^c \psi)^n.$$ This
measure does not put any mass on pluripolar sets. We recall the following important continuity property, proved in \cite{Bedford}.

\begin{theorem}[Bedford-Taylor] \label{thmconv} If $\psi_{i,k}$,
  $1\leq i\leq n+2,$ $k\in \mathbb{N},$ are sequences of locally bounded positive
  metrics such that each $\psi_{i,k}$ decreases to a locally bounded
  positive metric $\psi_i,$ then the signed measures
  $(\psi_{1,k}-\psi_{2,k})dd^c\psi_{3,k}\wedge ...\wedge dd^c\psi_{n+2,k}$ converge weakly to
  $(\psi_1-\psi_2)dd^c\psi_3\wedge ...\wedge dd^c\psi_{n+2}$. If each sequence of locally
  bounded positive metrics $\psi_{i,k}$ instead increase pointwise
  a.e. to a positive metric $\psi_i,$ then again the measures $(\psi_{1,k}-\psi_{2,k})dd^c\psi_{3,k}\wedge ...\wedge dd^c\psi_{n+2,k}$ converge weakly to $(\psi_1-\psi_2)dd^c\psi_3\wedge ...\wedge dd^c\psi_{n+2}$.
\end{theorem}

Since the curvature form $dd^c\phi$ of any smooth metric $\phi$ is a
representative of $c_1(L)$, we see that if $\phi_i$ is any $n$-tuple
of smooth metrics then
\begin{equation} \label{eqvolL} \int_X dd^c\phi_1\wedge ...\wedge
  dd^c\phi_n=\int_X c_1(L)^n
\end{equation}
which is just a topological invariant of $L$. Since any positive metric can be approximated from above in the manner of Theorem \ref{thmconv} by positive metrics that are smooth, we see that \eqref{eqvolL} still holds if the $\phi_i$ are merely assumed to be locally bounded instead of smooth.\medskip

Our proof that maximal envelopes are maximal (Theorem \ref{lemma3}) is based on an approximation argument that requires some technical results concerning convergence of plurisubharmonic functions.  Recall that a plurisubharmonic function is, by definition, upper semicontinuous, so
if $\psi$ is a positive metric then for each local frame $f$ the
function $\psi_f$ is upper semicontinuous. The plurifine topology is
defined as the coarsest topology in which all local plurisubharmonic functions
are continuous; a basis for this topology is given by sets of the form
$A\cap \{u>0\},$ where $A$ is open in the standard topology and
$u$ is a local plurisubharmonic function. This topology has the quasi-Lindel\"of property \cite[Thm 2.7]{Bedford2}, meaning that an arbitrary union of plurifine open sets differs from a countable subunion by at most a pluripolar set. Any basis set $A\cap \{u>0\}$ is Borel, so it follows from the quasi-Lindel\"of property that the plurifine open (and closed) sets lie in the completion of the Borel $\sigma$-algebra with respect to any Monge-Amp\`ere measure \cite[Prop 3.1]{Bedford2}.

\begin{definition}
A function $f$ is said to be \emph{quasi-continuous} on a set $\Omega$ if for every $\epsilon>0$ there exists an open set $U$ with capacity less than $\epsilon$ so that $f$ is continuous on $\Omega\setminus U$.
\end{definition}

We refer to \cite[(1.3)]{Bedford} for the definition of capacity, and in \cite[Thm 3.5]{Bedford} it is shown that plurisubharmonic functions are quasi-continuous.  \medskip

If $f_k$ is a sequence of non-negative continuous functions increasing to the characteristic function of an open set $A$ then the characteristic function of a basis set $A\cap \{u>0\}$ is the increasing limit of the non-negative quasi-continuous functions $$kf_{k}(\max\{u,0\}-\max\{u-1/k,0\}).$$ From this fact and the quasi-Lindel\"of property it follows that the characteristic function of any plurifine open set differs from an increasing limit of non-negative quasi-continuous functions at most on a pluripolar set. \medskip

A fundamental property of the Bedford-Taylor product is that it is local in the plurifine topology,
so if $\psi_i=\psi_i'$ for all $i$ on some plurifine open set $O$ then
$$\mathbf{1}_O dd^c\psi_1\wedge ...\wedge dd^c\psi_n=\mathbf{1}_O
dd^c\psi_1'\wedge ...\wedge dd^c\psi_n',$$ where $\mathbf{1}_O$
denotes the characteristic function of $O$. We also have that the convergence in Theorem \ref{thmconv} is local in this topology \cite[Thm 3.2]{Bedford2}, i.e. we get convergence when testing against bounded quasi-continuous functions. 

\begin{lemma} \label{lemmany1}
Let $\psi_k$ be a sequence of locally bounded positive metrics that decreases pointwise (or increases a.e.) to a locally bounded positive metric $\psi,$ and let $O$ be a plurifine open set. Then $$\mathbf{1}_O MA(\psi)\leq \liminf_{k\to \infty}\mathbf{1}_O MA(\psi_k),$$ where the $\liminf$ is to be understood in the weak sense, i.e. when testing against non-negative continuous functions.
\end{lemma}

\begin{proof}
Let $u_i$ be a sequence of quasi-continuous functions increasing to $\mathbf{1}_O$ except on a pluripolar set. Let $f$ be a non-negative continuous function. Since $u_iMA(\psi_k)$ converges weakly to $u_iMA(\psi),$ and $MA(\psi_k)$ does not put any mass on a pluripolar set, 
\begin{equation} \label{mnbv}
\int_X fu_iMA(\psi)=\lim_{k\to \infty}\int_X fu_iMA(\psi_k)\leq \liminf_{k\to \infty}\int_O fMA(\psi_k).
\end{equation}
Now $u_i$ increases to the characteristic function of $O$ except possibly on a pluripolar set, so letting $i$ tend to infinity in (\ref{mnbv}) yields $$\int_O fMA(\psi)\leq \liminf_{k\to \infty}\int_O fMA(\psi_k).$$ 
\end{proof}

For singular $\psi_i$ there is a (non pluripolar) product constructed
by Boucksom-Eyssi\-dieux-Guedj-Zeriahi \cite{Guedj}, building on a local construction due to Bedford-Taylor \cite{Bedford2}. Fix a locally bounded metric $\phi$, and
consider the auxiliary metrics $\psi_{i,k}:=\max\{\psi_i,\phi-k\}$ for
$k\in \mathbb{N},$ and the sets $O_k:=\bigcap_i \{\psi_i>\phi-k\}$. The
\emph{non-pluripolar product} of the currents $dd^c\psi_i,$ here denoted by
$dd^c\psi_1\wedge ...\wedge dd^c\psi_n$ is defined as the limit
$$dd^c\psi_1\wedge ...\wedge dd^c\psi_n:=\lim_{k
  \to\infty}\mathbf{1}_{O_k}dd^c\psi_{1,k}\wedge ...\wedge dd^c
\psi_{n,k}.$$ Since we are assuming that $X$ is compact this limit is
well defined \cite[Prop. 1.6]{Guedj}.  The (non-pluripolar)
Monge-Amp\`ere measure of a positive metric is $\psi$ is defined as
$MA(\psi):=(dd^c \psi)^n$. Essentially by construction, the
non-pluripolar product is local in the plurifine topology \cite[Prop.
1.4]{Guedj}, and is multilinear \cite[Prop 4.4]{Guedj}.

Clearly from the definition and (\ref{eqvolL}), for any
$n$-tuple of positive metrics $\psi_i$ on $L$, $$\int_X dd^c\psi_1\wedge
...\wedge dd^c\psi_n\leq \int_X c_1(L)^n,$$ however the inequality may
well be strict.  \medskip

Combining Lemma \ref{lemmany1} with the fact that the Monge-Amp\`ere measure is local in the plurifine topology yields the following continuity result.

\begin{lemma} \label{lemmabyt} Let $\psi_k$ be a sequence of positive
  metrics decreasing to a positive metric $\psi,$ and let $\phi$ be some locally bounded positive metric. If $O$ is a plurifine open set contained in $\{\psi>\phi-C\}$ for some constant $C$ then 
\begin{equation} \label{more}
  \mathbf{1}_O MA(\psi)\leq \liminf_{k\to
    \infty}\mathbf{1}_O MA(\psi_k),
\end{equation}
where again the $\liminf$ is to be understood in the weak sense. If $\psi_k$ instead is increasing a.e. to $\psi,$ and $O$ is a plurifine open set contained in $\{\psi_j>\phi-C\}$ for some natural number $j$ and some constant $C$ then once again $$\mathbf{1}_O MA(\psi)\leq \liminf_{k\to
    \infty}\mathbf{1}_O MA(\psi_k).$$
\end{lemma}

\begin{proof}
First assume that $\psi_k$ is decreasing to $\psi$. Let $\psi_k':=\max\{\psi_k,\phi-C\}$ and $\psi':=\max\{\psi,\phi-C\}$. From Lemma \ref{lemmany1} it follows that $$\mathbf{1}_O MA(\psi')\leq \liminf_{k\to
    \infty}\mathbf{1}_O MA(\psi_k'),$$ and since by assumption $\psi'=\psi$ and $\psi_k'=\psi_k$ on $O$ the lemma follows from the locality of the non-pluripolar product. The case where $\psi_k$ is increasing a.e. follows by the same reasoning.
\end{proof}

In \cite[Thm 1.16]{Guedj} it is shown that the non-pluripolar product, when restricted to metrics with small unbounded locus, has the following monotonicity property.

\begin{theorem} \label{monothm}
Let $\psi_i, \psi_i'$ be two $n$-tuples of positive metrics with small unbounded locus, and suppose that for all $i,$ $\psi_i$ is more singular than $\psi_i'$. Then $$\int_X dd^c\psi_1\wedge ...\wedge dd^c\psi_n\leq \int_X dd^c\psi_1'\wedge ...\wedge dd^c\psi_n'.$$
\end{theorem}

There is also a comparison principle for metrics with small unbounded locus \cite[Cor 2.3]{Guedj} and a domination principle \cite[Cor 2.5]{Guedj}. When combined with the comparison principle, the proof of the domination principle in \cite{Guedj} in fact yields a slightly stronger version:

\begin{theorem} \label{domination}
Let $\phi$ be a positive metric with small unbounded locus and suppose that there exists a positive metric $\rho,$ more singular than $\phi$, with small unbounded locus and such that $MA(\rho)$ dominates a volume form. If $\psi$ is a positive metric more singular than $\phi$ and such that $\psi\leq \phi$ a.e. with respect to $MA(\phi),$ then it follows that $\psi\leq \phi$ on the whole of $X$.
\end{theorem}

\subsection{The Aubin-Mabuchi Energy}\label{sectionaubinmabuchi}

The \emph{Aubin-Mabuchi energy bifunctional} maps any pair of
equivalent positive metrics $\psi_1$ and $\psi_2$ to the number
$$\mathcal{E}(\psi_1,\psi_2):=\frac{1}{n+1}\sum_{i=0}^n\int_X(\psi_1-\psi_2)(dd^c\psi_1)^i\wedge
(dd^c\psi_2)^{n-i}.$$ Observe $$\mathcal{E}(\psi+t,\psi)=t\int_X
MA(\psi).$$

The Aubin-Mabuchi energy restricted to the class of locally bounded
metrics has a cocycle property (see, for example, \cite[Cor
4.2]{Berman2}), namely if $\phi_0,$ $\phi_1$ and $\phi_2$ are locally
bounded equivalent metrics then
$$\mathcal{E}(\phi_0,\phi_2)=\mathcal{E}(\phi_0,\phi_1)+\mathcal{E}(\phi_1,\phi_2).$$

In fact the proof in \cite{Berman2} of the cocycle property extends to the case where the equivalent metrics are only assumed to have small unbounded locus, since the integration-by-parts formula of \cite{Guedj} used in the proof holds in that case.

This leads to an important monotonicity property.  If $\psi_0,\psi_1$
and $\psi_2$ are equivalent with small unbounded locus, and $\psi_0\geq \psi_1$, then
$$\mathcal{E}(\psi_0,\psi_2)\geq \mathcal{E}(\psi_1,\psi_2)$$ 
since $\mathcal{E}(\psi_0,\psi_2)=\mathcal{E}(\psi_0,\psi_1)+\mathcal{E}(\psi_1,\psi_2),$
and $\mathcal{E}(\psi_0,\psi_1)\ge 0$ as it is the
integral of the positive function $\psi_0-\psi_1$ against a positive
measure.

We also record the following lemma, which comes from the locality of
the non-pluripolar product in the plurifine topology.

\begin{lemma} \label{lemmavfr} Let $\psi_1\sim \psi_2$ be such that $\psi_1\geq \psi_2$. Let $\psi_1'$ and
  $\psi_2'$ be two other metrics such that $\psi_1'\sim \psi_2'$ and
  assume that $\{\psi_1'=\psi_2'\}=\{\psi_1=\psi_2\}$ and that $\psi_1'=\psi_1$ and $\psi_2'=\psi_2$ on the set where $\psi_1>\psi_2$. Then
  $$\mathcal{E}(\psi_1',\psi_2')=\mathcal{E}(\psi_1,\psi_2).$$
\end{lemma}

Following Phong-Sturm in \cite{Sturm4} we can relate weak geodesics to
the energy functional.  Let $I\subset \mathbb R$ be an interval and consider the annulus $A:=\{w\in \mathbb C: -\ln |w|\in I\}$ and $\pi\colon X\times A\to X$ be the projection.

\begin{definition}
  A curve of positive metrics $\phi_t,$ for $t\in I$ is said to be
  a \emph{weak subgeodesic} if there exists a locally bounded positive
  metric $\Phi$ on $\pi^*L$ that is rotation invariant and whose
  restriction to $X\times \{w\}$ equals $\phi_{-\ln |w|}$.   The curve $\phi_t$ is
  said to be a \emph{weak geodesic} if it is a weak subgeodesic and
  furthermore $\Phi$ solves the HCMA equation,
  i.e.\ $$MA(\Phi)=0$$ on $X\times A^{\circ}$.  A weak geodesic $\phi_t$ defined for $0\le t<\infty$ will be called a \emph{weak geodesic ray}.  
\end{definition}

As in the convex setting \eqref{volumeformula} there is a formula
\cite[Proposition 6.2]{Berman5} relating the Aubin-Mabuchi energy of a locally
bounded subgeodesic $\phi_t$ with the Monge-Amp\`ere measure of $\Phi,$
namely

\begin{equation} \label{volfor}
dd_t^c\mathcal{E}(\phi_t,\phi_a)=\pi_*(MA(\Phi)),
\end{equation}
where $\pi_*(MA(\Phi))$ denotes the push-forward of the measure $MA(\Phi)$ with respect to the projection $\pi$.   From this we immediately get the following lemma.

\begin{lemma} \label{lemmageo2} A curve $\phi_t$ of locally bounded
  positive metrics defined for $t\in [0,\infty)$ is a weak geodesic ray if and only if it is a
  subgeodesic and for any $a\in (0,\infty)$ the Aubin-Mabuchi energy
  $\mathcal{E}(\phi_t,\phi_a)$ is linear in $t$.
\end{lemma}

\section{Envelopes and maximal metrics}\label{sectionenvelopes}

In studying the Dirichlet problem for the HCMA equation it is often possible to give a solution as an envelope in some space
of plurisubharmonic functions (or positive metrics).  Such envelopes
will be crucial in our setting as well.

\begin{definition}
  If $\phi$ is a continuous metric, not necessarily positive, let
  $P\phi$ denote the envelope $$P\phi:=\sup\{\psi\leq \phi, \psi\in
  \PSH(L)\}.$$ Since $\phi$ is assumed to be continuous it follows
  that the upper semicontinuous regularization of $P\phi$ is bounded from above by $\phi$, and hence also by $P\phi$.  Thus $P\phi$ is itself upper semicontinuous and so $P\phi\in \PSH(L)$.
\end{definition}

The next theorem is essentially just a reformulation of a local result of Bedford-Taylor \cite[Corollary 9.2]{Bedford} in our global setting. It follows as a special case of \cite[Proposition 1.10]{Berman2} (letting $K=X$).
\begin{theorem} \label{theorem1} If $\phi$ is a continuous metric then $P\phi=\phi$ a.e. with respect to $MA(P\phi)$. 
\end{theorem}

Recall that if $A$ is a closed set and $\mu$ is a Borel measure we say that $\mu$ is said to be \emph{concentrated} on $A$ if $\mathbf{1}_A \mu=\mu,$ or equivalently $\mu(A^c)=0$. Thus another way of formulating Theorem \ref{theorem1} is to say that $MA(P\phi)$ is concentrated on $\{P\phi=\phi\}$. We now extend this result to more general envelopes that arise from the additional data of singularity type.

\begin{definition}
  Given a positive metric $\psi\in \PSH(L)$ let $P_{\psi}$ denote the
  projection operator on $\PSH(L)$ defined by $$P_{\psi}\phi:=
  \sup\{\psi' \leq \min\{\phi,\psi\}, \psi'\in \PSH(L)\}.$$ We also
  define $P_{[\psi]}$ by  $$P_{[\psi]}\phi:=\lim_{C\to
    \infty}P_{\psi+C}\phi=\sup\{\psi' \leq \phi, \psi'\sim \psi,
  \psi'\in \PSH(L)\}.$$
\end{definition}

Clearly $P_{\psi}\phi$ is monotone with respect to both $\psi$ and
$\phi$.  Since $\min\{\phi,\psi\}$ is upper semicontinuous, it follows
that the upper semicontinuous regularization of $P_{\psi}\phi$ is
still less than $\min\{\phi,\psi\},$ and thus $P_{\psi}\phi\in
\PSH(L)$. By this it follows that
$P_{\psi}(P_{\psi}\phi)=P_{\psi}\phi,$ i.e.\ that $P_{\psi}$ is indeed
a projection operator on $\PSH(L)$. One also notes that the upper
semicontinuous regularization of $P_{[\psi]}\phi,$ lies in $\PSH(L)$
and is bounded by $\phi$.

\begin{definition}\label{definitionmaximal}
  The \emph{maximal envelope} of $\phi$ with respect to the
  singularity type $[\psi]$ is defined to be
$$\phi_{[\psi]} := \usc(P_{[\psi]}\phi)$$
where $\usc$ denotes the process of taking the upper-semicontinuous regularization.
\end{definition}

\begin{definition} \label{defmax2}
  If $\psi\in \PSH(L),$ then $\psi$ is said to be
  \emph{maximal} with respect to a metric $\phi$ if $\psi\leq \phi$ and furthermore $\psi=\phi$ a.e.\ with respect to $MA(\psi)$. Similarly, if $A$ is a measurable set, we say that $\psi$ is maximal with respect to $\phi$ on $A$ if $\psi\leq \phi$ and $\psi=\phi$ a.e.\ on $A$ with respect to $MA(\psi)$.
\end{definition}

\begin{remark}\label{rmk:singularitytype}
The terminology is justified by a proof below that the maximal envelope of a continuous metric $\phi$ is maximal with respect to $\phi$.   As it is defined as a limit, it is not clear from the definition if $\phi_{[\psi]}$ is equivalent to $\psi$ (this can be shown if $\psi$ has algebraic singularities by passing to a suitable resolution, and we refer the reader to \cite{Ross3} for a further study of maximal envelopes).   For this reason,  the method in the proof of Theorem \ref{theorem1} in \cite{Berman2} does not directly apply, so instead we will use an approximation argument.     
\end{remark}

Our use of the word maximal is motivated by the following property:

\begin{proposition} \label{maximality}
Let $\psi$ be maximal with respect to a metric $\phi$.   Suppose also that there exists a positive metric $\rho\succeq \psi$ with small unbounded locus and such that $MA(\rho)$ dominates a volume form.   Then for any $\psi'\sim \psi$ with $\psi\le \phi$ we have $\psi'\leq \psi$.
\end{proposition}

\begin{proof}
Since $\psi'\leq \phi$, the maximality assumption yields $\psi'\leq \psi$ a.e.\ with respect to $MA(\psi),$ so the proposition thus follows from the domination principle (Theorem \ref{domination}).
\end{proof}

The next two lemmas are the main steps in showing that maximal envelopes are maximal.

\begin{lemma} \label{lemlim} Let $\psi_k$ be a sequence of positive
  metrics increasing a.e.\ to a positive metric $\psi,$ and assume that
  all $\psi_k$ are maximal with respect to a fixed continuous metric $\phi$ on some plurifine open set $O$.
  Then $\psi$ is maximal with respect to $\phi$ on $O$.
\end{lemma}

\begin{proof}
Since $\phi$ was assumed to be continuous, $\psi\leq \phi$. Now, for all $k$
  $$\{\psi_k=\phi\}\subseteq \{\psi=\phi\}$$ and thus by the the
  maximality of $\psi_k$, we know $\mathbf{1}_O MA(\psi_k)$ is concentrated on $\{\psi=\phi\}$. Since $\psi\leq \phi$ we have that $\{\psi=\phi\}=\{\psi\geq \phi\},$ and since $\phi$ is continuous this is a closed set. 
Let $C$ be a constant. The set $O\cap \{\psi_1>\phi-C\}$ is plurifinely open, so by Lemma \ref{lemmabyt}  it follows that 
\begin{equation} \label{poil}
\mathbf{1}_O  \mathbf{1}_{\{\psi_1>\phi-C\}}MA(\psi)\leq \liminf_{k\to\infty} \mathbf{1}_O \mathbf{1}_{\{\psi_1>\phi-C\}}MA(\psi_k).
\end{equation}
It is easy to see that if $\mu_k$ is a sequence of measures all concentrated on a closed set $A$, and $$\mu\leq \liminf_{k\to \infty}\mu_k$$ in the weak sense, then $\mu$ is also concentrated on $A$. It thus follows from (\ref{poil}) that $\mathbf{1}_O  \mathbf{1}_{\{\psi_1>\phi-C\}}MA(\psi)$ is concentrated on $\{\psi=\phi\}$. Since $MA(\psi)$ puts no mass on the pluripolar set $\{\psi_1=-\infty\}$ the lemma follows by letting $C$ tend to infinity. 
\end{proof}

\begin{lemma} \label{xcv}
Let $\psi\in \PSH(L)$ and let $\phi$ be a continuous metric. Then the envelope $P_{\psi}\phi$ is maximal with respect to $\phi$ on the plurifine open set $\{\psi>\phi\}$.
\end{lemma}

\begin{proof}
Clearly $P_{\psi}\phi\le \phi$, so we have to show $P_{\psi}\phi$ is equal to $\phi$ almost everywhere on $\{\psi>\phi\}$ with respect to $MA(P_{\psi}\phi)$.  Let $\phi_k$ be a sequence of
  continuous metrics decreasing pointwise to $\min\{\phi,\psi\},$ so that
  $\phi_k\le \phi$ for all $k$ and $\phi_k=\phi$ on the set $\{\psi>\phi\}.$ For example let $\phi_k:=\min\{\phi,\psi_k\}$ where $\psi_k$ is a sequence of smooth metrics decreasing pointwise to $\psi$. From Theorem \ref{theorem1} it follows that $MA(P\phi_k)$ is concentrated on $\{P\phi_k=\phi_k\},$ and since $\phi_k=\phi$ when $\psi>\phi$ we see $\mathbf{1}_{\{\psi>\phi\}}MA(P\phi_k)$ is concentrated on $\{P\phi_k=\phi\}$. Now $P\phi_{k}$ is decreasing in $k$ and $\lim_{k\to
    \infty}P\phi_k\leq \min\{\phi,\psi\}$. At the same time, for any $k\in \mathbb{N}$ we clearly have that $P_{\psi}\phi\leq P\phi_k,$ which taken together means that
$$\lim_{k\to \infty}P\phi_k=P_{\psi}\phi.$$ Since $P\phi_k\leq \phi$ this implies that $\{P\phi_k=\phi\}$ is decreasing in $k$ and
\begin{equation} \label{nuenu}
\{ P_\psi \phi = \phi\}=\bigcap_{k\in \mathbb{Z}}\{P\phi_k =\phi\}.
\end{equation}
Let $O$ denote the plurifine open set $\{\psi>\phi\}\cap \{P_{\psi}\phi>\phi-C\}$. By Lemma \ref{lemmabyt} we know $$\mathbf{1}_O MA(P_{\psi}\phi)\leq \liminf_{k\to \infty}\mathbf{1}_O MA(P\phi_k),$$ and thus we conclude that $\mathbf{1}_O MA(P_{\psi}\phi)$ is concentrated on $\{P\phi_k=\phi\}$ for any $k,$ so by \eqref{nuenu} we get that $\mathbf{1}_O MA(P_{\psi}\phi)$ is concentrated on $\{P_\psi \phi = \phi\}$. Since $MA(P_{\psi}\phi)$ puts no mass on the pluripolar set $\{P_{\psi}\phi=-\infty\},$ letting $C$ tend to infinity yields the lemma.
\end{proof}

\begin{theorem} \label{lemma3} Let $\psi\in \PSH(L)$ and let $\phi$ be
  a continuous metric. Then $\phi_{[\psi]}$ is maximal
  with respect to $\phi,$ i.e.\ $\phi_{[\psi]}=\phi$ a.e.\ with respect to $MA(\phi_{[\psi]})$.
\end{theorem}

\begin{proof}
$P_{[\psi]}\phi=\phi_{[\psi]}$ a.e., and since $P_{\psi+C}\phi$
increases to $P_{[\psi]}\phi,$ it thus increases to $\phi_{[\psi]}$
a.e.. By Lemma \ref{xcv} we get that $P_{\psi+C}\phi$ is maximal with respect to $\phi$ on the plurifine open set $\{\psi>\phi-C\}$ and thus also on any set $\{\psi>\phi-C'\}$ whenever $C'\leq C$. From Lemma \ref{lemlim} it thus follows that $\phi_{[\psi]}$ is maximal with respect to $\phi$ on the set $\{\psi>\phi-C\}$ for any $C$. Since $MA(\phi_{[\psi]})$ puts no mass on $\{\psi=-\infty\}$ the theorem follows.
\end{proof}

\begin{example}
  Consider the case that $s$ is a section of $rL$ that vanishes along
  a divisor $D$, and set $\psi = \frac{1}{r} \ln |s|^2$.  Then the
  maximal envelope $\phi_{[\psi]}$ is considered by Berman \cite[Sec.
  4]{Berman4}, and equals
$${\sup}^* \{\psi' \le \phi : \psi'\in PSH(L), \nu_D(\psi')\ge 1\}$$
where $\nu_D$ denotes the Lelong number along $D$.  This metric
governs the Bergman kernel asymptotics of sections of $kL$ for $k\gg
0$ that vanish along the divisor $D$. The more general case when $\psi$ has analytic singularities is also considered in \cite{Berman4}. 
\end{example}

The maximal property gives the following bounds on the energy
functional which will be crucial for our construction of weak
geodesics (Theorem \ref{theorem2}).

\begin{proposition} \label{key} Suppose that $\psi$ is maximal with respect to a positive metric $\phi$ with small unbounded locus, and let $t>0$. Then 
  % we have the following upper and lower bound on the Aubin-Mabuchi
  % energy of $\max\{\psi+t,\phi\}$ relative to $\phi,$
  \begin{equation} \label{ineqdre} t\int_X MA(\psi)\leq
    \mathcal{E}(\max\{\psi+t,\phi\},\phi)\leq t\int_X MA(\phi).
  \end{equation} 
\end{proposition}

\begin{proof}
  Since by assumption $\psi\leq \phi$ we know
  $\max\{\psi+t,\phi\}\leq \phi+t,$ so from the monotonicity of the
  Aubin-Mabuchi energy it follows that
  $$\mathcal{E}(\max\{\psi+t,\phi\},\phi)\leq
  \mathcal{E}(\phi+t,\phi)=t\int_X MA(\phi)$$ which gives the upper
  bound.  For the lower bound, first choose an $\epsilon$ with
  $0<\epsilon<t$. Again by monotonicity,
  \begin{equation} \label{eqdre}
    \mathcal{E}(\max\{\psi+t,\phi\},\phi)\geq
    \mathcal{E}(\max\{\psi+t,\phi\},\max\{\psi+\epsilon,\phi\}).
  \end{equation} 

Now clearly
  \begin{equation} \label{eqdre2}
    \mathcal{E}(\max\{\psi+t,\phi\},\max\{\psi+\epsilon,\phi\})\geq
    (t-\epsilon)\int_{\{\psi+\epsilon>\phi\}}MA(\psi).
  \end{equation}
  By the assumption that $\psi$ is maximal with respect to $\phi$ 
$$\int_{\{\psi=\phi\}}MA(\psi)=\int_X MA(\psi)$$ and
  since $\{\psi=\phi\}\subseteq \{\psi+\epsilon>\phi\}$, the
  combination of (\ref{eqdre}) and (\ref{eqdre2}) yields
  $$\mathcal{E}(\max\{\psi+t,\phi\},\phi)\geq (t-\epsilon)\int_X
  MA(\psi).$$ Since $\epsilon>0$ was chosen arbitrarily the lower
  bound in (\ref{ineqdre}) follows.
\end{proof}

\section{Test curves and analytic test
  configurations}\label{section:testcurves}

\begin{definition}\label{definitiontestcurve}
  A map $\lambda\mapsto \psi_{\lambda}$ from $\mathbb{R}$ to $\PSH(L)$
  is called a \emph{test curve} if there is a constant $C$ such that
  \begin{enumerate}
  \item[(i)] $\psi_{\lambda}$ is equal to some locally bounded
    positive metric $\psi_{-\infty}$ for $\lambda<-C$,
  \item[(ii)] $\psi_{\lambda}\equiv -\infty$ for $\lambda>C$,
  \item[(iii)] $\psi_{\lambda}$ has small unbounded locus whenever $\psi_{\lambda}\not \equiv -\infty,$ and
  \item[(iiii)] $\psi_{\lambda}$ is concave in $\lambda$.
  \end{enumerate}
\end{definition}

  Observe that since $\psi_{\lambda}$ is concave and constant for
  $\lambda$ sufficiently negative it is decreasing in $\lambda$.    The set of test curves forms a convex set, by letting
$$(\sum r_i\gamma_i)(\lambda):=\sum r_i\gamma_i(\lambda).$$ It is also
clear that any translate $\gamma_a(\lambda):=\gamma(\lambda-a)$ of a
test curve $\gamma$ is a new test curve.

We introduce the notation $\lambda_c$ for the critical value of a test curve defined as $$\lambda_c:=\inf\{\lambda: \psi_{\lambda}\equiv -\infty\}.$$ 
For later use we record here two continuity properties of test curves.

\begin{lemma} \label{leftcontlemma}\
  \begin{enumerate}\item 
    A test curve $\psi_{\lambda}$ is left-continuous in $\lambda$ as long as 
    $\lambda<\lambda_c$.

  \item Suppose that $\lambda<\lambda_c$ and $\lambda_k$ is a decreasing sequence that tends to $\lambda$.  Then
    \begin{equation}\label{rightcont}
{\lim}^*_{k\to \infty} \psi_{\lambda_k} = \psi_{\lambda}. 
\end{equation}
(So a a test curve is right continuous modulo taking an upper semicontinuous regularization.)
  \end{enumerate}

\end{lemma}

\begin{proof}
For (1), let $\lambda_k$ increase to some $\lambda<\lambda_c,$ so we need to show that $$\lim_{k\to \infty}\psi_{\lambda_k}=\psi_{\lambda}.$$ By our hypothesis there exists a $\lambda'$ such that $\lambda<\lambda'<\lambda_c,$ and thus $\psi_{\lambda'}\not \equiv -\infty$. Since $\psi_{\lambda}(x)$ is concave in $\lambda$ it is continuous for all $x$ such that $\psi_{\lambda'}(x)\neq -\infty$. Thus $\psi_{\lambda_k}$ converges to $\psi_{\lambda}$ at least away from a pluripolar set, i.e.\ a.e.\ with respect to a volume form. On the other hand we have that $\psi_{\lambda_k}$ is decreasing in $k,$ so the limit is a positive metric. Now if two positive metrics coincide a.e.\ with respect to a volume form it follows that they are equal everywhere, because this is true locally for plurisubharmonic function \cite[Corollary 2.9.8]{Klimek}.

The proof of (2) is essentially the same.  This time $\lambda_k$ is a decreasing sequence, so as $\lambda<\lambda_c$ we may as well assume that each $\lambda_k<\lambda'$ and so in particular $\psi_{\lambda_k}\not\equiv -\infty$.  Then the $\psi_{\lambda_k}$ form an increasing sequence so the left hand side of \eqref{rightcont} is a positive metric.   But for the same reason as above, the limit $\lim_{k\to \infty} \psi_{\lambda_k}$ equals $\psi_{\lambda}$ away from a pluripolar set, and thus the left and right hand side of \eqref{rightcont} agree a.e. with respect to a volume form, and thus are equal everywhere.
\end{proof}

\begin{definition}\label{definitionanalytictestconfig}
  A map $\gamma$ from $\mathbb{R}$ to $Sing(L)$ is called an
  \emph{analytic test configuration} if it is the composition of a
  test curve with the natural projection of $\PSH(L)$ to $Sing(L)$.
\end{definition}
As with the set of test curves, the set of analytic test
configurations is convex.  We now extend the definition of the
maximal envelope (Definition \ref{definitionmaximal}) to test
curves.

\begin{definition} \label{lemma4} Let $\psi_{\lambda}$ be a test curve
  and $\phi$ an element in $\mathcal{H}(L)$. The \emph{maximal
    envelope} of $\phi$ with respect to $\psi_{\lambda}$ is the map
$$\lambda \mapsto \phi_{\lambda} := \phi_{[\psi_{\lambda}]}=\usc (P_{[\psi_{\lambda}]}\phi).$$ 
\end{definition}

It is easy to see that $\phi_{\lambda}$ only depends on $\phi$ and the
analytic test configuration $[\psi_{\lambda}],$ since if
$\psi_{\lambda}'\sim \psi_{\lambda}$ we trivially have
$\phi_{[\psi_{\lambda}]}=\phi_{[\psi_{\lambda}']}$. Observe also that
since $\psi_{-\infty}$ is locally bounded, we have $\phi_{\lambda} =
\phi$ for $\lambda<-C$.

\begin{definition}
  We say that a test curve $\psi_{\lambda}$ is \emph{maximal} if for
  all $\lambda$ the metric $\psi_{\lambda}$ is maximal with respect
  to $\psi_{-\infty}$.
\end{definition}

Since $\psi_{\lambda}$ is decreasing in $\lambda,$
$$\{\psi_{\lambda'}=\psi_{\lambda}\}\supseteq
\{\psi_{\lambda'}=\psi_{-\infty}\}\quad \text{ if} \quad \lambda\leq \lambda'.$$ It
follows that if $\psi_{\lambda}$ is a maximal test curve,
$\psi_{\lambda'}$ is maximal with respect to $\psi_{\lambda}$
whenever $\lambda\leq \lambda'$.   We shall show in the next section how the Legendre transform of a maximal test curve gives a weak geodesic ray, and end this section by showing how maximal envelopes give rise to maximal test curves.

\begin{proposition}
  The maximal envelope $\phi_{\lambda}$ is a maximal test curve.
\end{proposition}

\begin{proof}

  We first show it is a test curve.  Pick a real number $C$. Let
  $\lambda$ and $\lambda'$ be two real numbers, and let $0\leq t\leq
  1$. By the concavity of $\psi_{\lambda}$,
  $$tP_{\psi_{\lambda}+C}\phi+(1-t)P_{\psi_{\lambda'}+C}\phi\leq
  t\psi_{\lambda}+(1-t)\psi_{\lambda'}+C\leq
  \psi_{t\lambda+(1-t)\lambda'}+C.$$ Thus from the definition of the
  projection operator,
  $$tP_{\psi_{\lambda}+C}\phi+(1-t)P_{\psi_{\lambda'}+C}\phi\leq
  P_{\psi_{t\lambda+(1-t)\lambda'}+C}\phi,$$ which means that
  $P_{\psi_{\lambda}+C}\phi$ is concave in $\lambda$ for all $C$.
  Since $P_{\psi_{\lambda}+C}\phi$ increases to
  $P_{[\psi_{\lambda}]}\phi,$ and an increasing sequence of concave
  functions is concave, we get that $P_{[\psi_{\lambda}]}\phi$ is
  concave, and because of the monotonicity of the upper semicontinuous
  regularization it follows that
  $\usc(P_{[\psi_{\lambda}]}\phi)=\phi_{\lambda}$ also is concave. The
  other properties of a test curve are immediate.

  Clearly $\phi_{-\infty}=\phi,$ so that $\phi_{\lambda}$ is maximal
  follows from Theorem \ref{lemma3}.
\end{proof}

\section{The Legendre transform and geodesic
  rays}\label{sectionlegendre}

If $f$ is a convex function in the real variable $\lambda,$ the set of
subgradients of $f,$ denoted by $\Delta_f,$ is the set of $t\in
\mathbb{R}$ such that $f(\lambda)-t\lambda$ is bounded from below. If
$f$ happens to be differentiable, then the set subgradients
coincides with the image of the derivative of $f$. By convexity of
$f$, the set of subgradients is convex, i.e.\ an interval. Recall
that the Legendre transform of $f,$ here denoted by $\widehat{f}$, is
the function on $\Delta_f$ defined as
$$\widehat{f}(t):=\sup_{\lambda}\{t\lambda-f(\lambda)\}.$$ Since
$\widehat{f}$ is defined as the supremum of the linear functions
$t\lambda-f(\lambda),$ it follows that $\widehat{f}$ is convex.

If $f$ is concave then of course $-f$ is convex,
and one can define the Legendre transform of $f,$ also denoted by
$\widehat{f},$ as the Legendre transform of $-f,$ i.e.
$$\widehat{f}(t):=\sup_{\lambda}\{f(\lambda)+t\lambda\},$$ which is
thus convex.

\begin{definition}
  The \emph{Legendre transform} of a test curve $\psi_{\lambda}$ is 
  $$\widehat{\psi}_t:={\sup}^*_{\lambda\in
    \mathbb{R}}\{\psi_{\lambda}+t\lambda\},$$ where
  $t\in[0,\infty)$.
\end{definition}
Recall that the star denotes the operation of  taking the upper
semicontinuous regularization.

\begin{lemma} \label{lemma8} Let $\psi_{\lambda}$ be any test curve
  (not necessarily maximal).  Then the Legendre transform
  $\widehat{\psi}_t$ is locally bounded for all $t$, and the map
  $t\mapsto \widehat{\psi}_t$ is a subgeodesic ray emanating from
  $\psi_{-\infty}$.
\end{lemma}
\begin{proof}
  By assumption, for some $\lambda,$ $\psi_{\lambda}$ is locally
  bounded, and trivially $\widehat{\psi}_t\geq
  \psi_{\lambda}+t\lambda,$ thus $\widehat{\psi}_t$ is locally
  bounded.  It is clear that for a fixed $\lambda,$ the curve
  $\psi_{\lambda}+t\lambda$ is a subgeodesic. Clearly $\sup_{\lambda\in
    \mathbb{R}}\{\psi_{\lambda}+t\lambda\}$ is convex and Lipschitz in $t,$ and the same is easily seen to hold for $\widehat{\psi}_t$. Thus $\widehat{\psi}_t$ is upper semicontinuous in the directions in $X$ and also Lipschitz in $t,$ which implies that it is upper semicontinuous on the product $X\times \mathbb R_{\ge 0}$ . Therefore $\widehat{\psi}_t$ (thought of as a function on the product) coincides with the upper semicontinuous regularization of of $\sup_{\lambda\in
    \mathbb{R}}\{\psi_{\lambda}+t\lambda\}$.  

Now, taking the upper
  semicontinuous regularization of the supremum of subgeodesics yields
  a subgeodesic, as long as it is bounded from above. We observed
  above that $\psi_{\lambda}\leq \psi_{-\infty}$. Now for some
  constant $C$, $\psi_{C}\equiv -\infty$. It follows that
  $\widehat{\psi}_t\leq \psi_{-\infty}+tC,$ so it is bounded from
  above and thus it is a subgeodesic.

  Finally by definition
  $\widehat{\psi}_0={\sup}_{\lambda}^*\{\psi_{\lambda}\},$ which clearly
  is equal to $\psi_{-\infty}$ since $\psi_{\lambda}\leq
  \psi_{-\infty}$ ($\psi_{\lambda}$ being decreasing in $\lambda$) and
  $\psi_{-\infty}$ itself being upper-semicontinuous.
\end{proof}

One can also consider the inverse Legendre transform, going from subgeodesic rays to concave curves of positive metrics. 

\begin{definition}
The Legendre transform of a subgeodesic ray $\phi_t,$ $t\in [0,\infty),$ denoted by $\widehat{\phi}_{\lambda},$ $\lambda\in \mathbb{R},$ is defined as $$\widehat{\phi}_{\lambda}:=\inf_{t\in [0,\infty)}\{\phi_t-t\lambda\}.$$
\end{definition} 

\begin{remark}\label{rmk:Berndtsson}
  It follows from Kiselman's minimum principle (see \cite{Kiselman}) that for any $\lambda\in \mathbb{R},$ $\widehat{\phi}_{\lambda}$ is a positive metric (we would like to thank Bo Berndtsson for this observation). Furthermore it is clear that $\widehat{\phi}_{\lambda}$ is concave and decreasing in $\lambda$. From the involution property of the (real) Legendre transform it follows that the Legendre transform of $\widehat{\phi}_{\lambda}$ is $\phi_t,$ thus any subgeodesic ray is the Legendre transform of a concave curve of positive metrics.
\end{remark}

The goal of this section is to prove that if $\psi_{\lambda}$ is an
\emph{maximal} test curve then the Legendre transform
$\widehat{\psi}_t$ of $\psi_{\lambda}$ is a weak geodesic ray
emanating from $\psi_{-\infty}$. By Lemma \ref{lemma8} we know $\widehat{\psi}_t$ is a subgeodesic ray emanating from $\psi_{-\infty}$. What remains then is to show that if $\psi_{\lambda}$ is maximal then the Aubin-Mabuchi energy $\mathcal{E}(\widehat{\psi}_t,\widehat{\psi}_0)$ is linear in $t$, which we now do with an approximation argument. \medskip

For $N\in \mathbb{N}$ consider the approximation
$\widehat{\psi}_t^{N}$ to $\widehat{\psi}_t,$ given by
$$\widehat{\psi}_t^{N}:=\sup_{k\in
  \mathbb{Z}}\{\psi_{k2^{-N}}+tk2^{-N}\}.$$ Since $\psi_{\lambda}$ is
concave it is continuous in $\lambda$ at all points such that
$\psi_{\lambda}(x)>-\infty$. From the continuity it follows that
$\widehat{\psi}_t^{N}$ will increase pointwise to $\widehat{\psi}_t$
a.e.\ as $N$ tends to infinity.  Also let $\widehat{\psi}_t^{N,M}$
denote the curve $$\widehat{\psi}_t^{N,M}:=\sup_{k\in \mathbb{Z},
  k\leq M}\{\psi_{k2^{-N}}+tk2^{-N}\},$$ 
and observe that $\widehat{\psi}_t^N$ and $\widehat{\psi}_t^{N,M}$ are all locally
bounded.

\begin{lemma} \label{lemma5} Let $M<M'$ be two integers. Then
$$\widehat{\psi}_t^{N,M'}= \psi_{M'2^{-N}}+tM'2^{-N}$$ implies that $$\widehat{\psi}_t^{N,M}= \psi_{M2^{-N}}+tM2^{-N}.$$
\end{lemma}

\begin{proof}
  Certainly $f(\lambda):=\psi_{\lambda}(x)+t\lambda$ is concave in
  $\lambda$. If $$\widehat{\psi}_t^{N,M}> \psi_{M2^{-N}}+tM2^{-N}$$ at
  $x,$ then $f$ would be strictly decreasing at $\lambda=M2^{-N},$ so
  by concavity we would get that
  $f(M'2^{-N})<f(M2^{-N})<\widehat{\psi}_t^{N,M}(x),$ which would be a
  contradiction.
\end{proof}

\begin{lemma} \label{lemma6} If $\psi_{\lambda}$ is a maximal test
  curve then
  \begin{equation*}
    t2^{-N}\int_X MA(\psi_{(M+1)2^{-N}})\leq \mathcal{E}(\widehat{\psi}_t^{N,M+1},\widehat{\psi}_t^{N,M})\leq t2^{-N}\int_X MA(\psi_{M2^{-N}}).
  \end{equation*}
\end{lemma}

\begin{proof}
  By Lemma \ref{lemma5} it follows that
  $\widehat{\psi}_t^{N,M}=\psi_{M2^{-N}}+tM2^{-N}$ on the support of
  $\widehat{\psi}_t^{N,M+1}-\widehat{\psi}_t^{N,M}$ and thus Lemma
  \ref{lemmavfr} yields
  \begin{equation} \label{equation1}
    \mathcal{E}(\widehat{\psi}_t^{N,M+1},\widehat{\psi}_t^{N,M})=\mathcal{E}(\max\{\psi_{M2^{-N}},\psi_{(M+1)2^{-N}}+t2^{-N}\},\psi_{M2^{-N}}).
  \end{equation} 
  Since we assumed that $\psi_{\lambda}$ is maximal,
  $\psi_{(M+1)2^{-N}}$ is maximal with respect to $\psi_{M2^{-N}},$
  and thus the lemma follows immediately from Lemma \ref{key}.
\end{proof}

Let $\psi_{\lambda}$ be a maximal test curve, and let $F(\lambda)$
denote the function $$F(\lambda):=\int_X MA(\psi_{\lambda}).$$ Whenever $\lambda<\lambda',$ $\psi_{\lambda'}\leq \psi_{\lambda}$ and therefore it follows from Theorem \ref{monothm} that $F(\lambda)$ is decreasing in $\lambda,$ hence $F(\lambda)$ is Riemann integrable.

\begin{proposition} \label{lemma9} If $\psi_{\lambda}$ is a maximal
  test curve then
  \begin{equation} \label{ohyes}
    \mathcal{E}(\widehat{\psi}_t,\widehat{\psi}_0)=-t\int_{\lambda=-\infty}^{\infty}\lambda
    dF(\lambda).
  \end{equation}
\end{proposition}

\begin{proof}
  Suppose first $m\in \mathbb{Z}$ is such that
  $\psi_{m}=\psi_{-\infty}$. For a given $N\in \mathbb N$ set
  $M=m2^N$. Then
  $$\widehat{\psi}_t^{N,M}=\psi_{-\infty}+tm=\widehat{\psi}_0+tm.$$

  By repeatedly using the cocycle property of the Aubin-Mabuchi energy
  in combination with Lemma \ref{lemma6} we get that
  \begin{equation} \label{eqwer} t\sum_{k> M}2^{-N}F((k+1)2^{-N})\leq
    \mathcal{E}(\widehat{\psi}_t^N,\widehat{\psi}_t^{N,M})\leq
    t\sum_{k>M}2^{-N}F(k2^{-N}).
  \end{equation}
  We noted above that $\widehat{\psi}_t^{N}$ increases pointwise to
  $\widehat{\psi}_t$ a.e.\ as $N$ tends to infinity. By the continuity
  of the Aubin-Mabuchi energy under a.e.\ pointwise increasing
  sequences \eqref{lemmabyt},
  $$\mathcal{E}(\widehat{\psi}_t,\widehat{\psi}_0+tm)=t\int_{\lambda=m}^{\infty}\lambda
  F(\lambda)d\lambda,$$ since both the left- and the right-hand side
  of (\ref{eqwer}) converges to this. Again using the cocycle property
  we get that
  \begin{eqnarray}
    \mathcal{E}(\widehat{\psi}_t,\widehat{\psi}_0)=\mathcal{E}(\widehat{\psi}_t,\widehat{\psi}_0+tm)+\mathcal{E}(\widehat{\psi}_0+tm,\widehat{\psi}_0)= \nonumber \\=t\int_{\lambda=m}^{\infty}\lambda F(\lambda)d\lambda+tm\int_X MA(\psi_{-\infty})=t\int_{\lambda=m}^{\infty}F(\lambda)d\lambda+tmF(m). \label{oneone}
  \end{eqnarray}
  Since by our assumption the measure $dF$ is zero on $(-\infty,m)$,
  integration by parts yields
  \begin{eqnarray}
    -t\int_{\lambda=-\infty}^{\infty}\lambda dF(\lambda)=-t\lambda F(\lambda)|_m^{\infty}+t\int_{\lambda=m}^{\infty}F(\lambda)d\lambda=\nonumber \\=tmF(m)+t\int_{\lambda=m}^{\infty}F(\lambda)d\lambda. \label{twotwo}
  \end{eqnarray} 
  The proposition follows from combining equation (\ref{oneone}) and
  equation (\ref{twotwo}).
\end{proof}

\begin{theorem} \label{theorem2} The Legendre transform
  $\widehat{\psi}_t$ of a maximal test curve $\psi_{\lambda}$ is a
  weak geodesic ray emanating from $\psi_{-\infty}$.
\end{theorem}

\begin{proof}
  That $\widehat{\psi}_t$ is a subgeodesic emanating from
  $\psi_{-\infty}$ was proved in Lemma \ref{lemma8}. According to
  Proposition \ref{lemma9} the energy
  $\mathcal{E}(\widehat{\psi}_t,\widehat{\psi}_0)$ is linear in $t,$
  and therefore by Lemma \ref{lemmageo2} we get that $\widehat{\psi}_t$ is
  a weak geodesic ray.
\end{proof}

These weak geodesics are continuous in $\phi$ in the following sense:

\begin{proposition}
  Let $\psi_{\lambda}$ be a test curve and $\phi,\phi'\in \mathcal
  H(L)$.  Suppose $\phi_{\lambda}$ is the maximal curve of $\phi$
  (with respect to $\psi_{\lambda}$) and similarly for
  $\phi_{\lambda}'$.  If $\|\phi-\phi'\|_{\infty}<C$ then
  $$\|\widehat{\phi}_t-\widehat{\phi'}_t\|_{\infty}<C\quad \text{for
    all }t.$$
\end{proposition}
\begin{proof}
  We claim that $\|\phi_{\lambda}-\phi'_{\lambda}\|_{\infty}<C$ for
  all $\lambda$.  But this is clear since $\phi\leq \phi'$ implies
  that $\phi_{\lambda}\leq \phi'_{\lambda}$ for all $\lambda$. It is
  also clear that $(\phi+C)_{\lambda}=\phi_{\lambda}+C$ when $C$ is a
  constant.    Thus $\widehat{\phi}_t\leq \widehat{\phi'}_t$ for all $t$ and $(\phi+C)_{\lambda}=\phi_{\lambda}+C$, so consequently $\widehat{\phi+C}_t=\widehat{\phi}_t+C$  which proves the lemma.
\end{proof}

Let $[\psi_{\lambda}]$ be an analytic test configuration, and let $\phi_{\lambda}$ be the associated maximal test curve. Then $[\phi_{\lambda}]$ defines a new analytic test configuration. This could possibly differ from $[\psi_{\lambda}],$ but the following proposition tells us that the associated geodesic rays are the same.

\begin{proposition} \label{vasd}
Let $\phi'\in \mathcal{H}(L)$. Then the Legendre transform of $\phi'_{[\phi_{\lambda}]}$ coincides with the Legendre transform of $\phi'_{\lambda}:=\phi'_{[\psi_{\lambda}]}$.
\end{proposition} 

\begin{proof}
Since $\phi'_{\lambda}\sim \phi_{\lambda}$ we get that $\phi'_{[\phi_{\lambda}]}=\phi'_{[\phi'_{\lambda}]},$ thus without loss of generality we can assume that $\phi'=\phi$. Recall that the critical value $\lambda_c$ was defined as $$\lambda_c:=\inf \{\lambda: \phi_{\lambda}\equiv -\infty\}.$$ If $\lambda<\lambda_c$ there exists a $\lambda'$ such that $\lambda<\lambda'<\lambda_c,$ and thus by the assumption $\phi_{\lambda'}$ has small unbounded locus. Let $C$ be a constant less than $\lambda$ such that $\phi_C=\phi$. By concavity it follows that 
$$\phi_{\lambda}\geq r\phi+(1-r)\phi_{\lambda'},$$ 
where $0<r<1,$ is chosen such that $$\lambda=rC+(1-r)\lambda'.$$ If we let $$\rho:=r\phi+(1-r)\phi_{\lambda'},$$ by the multilinearity of the Monge-Amp\'ere operator it follows that $MA(\rho)$ dominates the volume form $r^nMA(\phi)$. Furthermore $\rho$ has small unbounded locus and is more singular than $\phi_{\lambda}$. Thus by Proposition \ref{maximality} we have $$P_{\phi_{\lambda}+C}\phi\leq \phi_{\lambda}$$ for any constant $C$ and therefore 
\begin{equation} \label{equationzxc}
\phi_{[\phi_{\lambda}]}=\phi_{\lambda},
\end{equation}
whenever $\lambda<\lambda_c$. If $\lambda>\lambda_c$ then clearly equation (\ref{equationzxc}) holds as well since both sides are identically equal to minus infinity. It follows that for any $\epsilon>0,$  $$\phi_{\lambda}\leq\phi_{[\phi_{\lambda}]}\leq \phi_{\lambda-\epsilon},$$ which implies that $$\widehat{(\phi_{\lambda})}_t\leq \widehat{(\phi_{[\phi_{\lambda}]})}_t\leq \widehat{(\phi_{\lambda-\epsilon})}_t=\widehat{(\phi_{\lambda})}_t+\epsilon t.$$ Since $\epsilon>0$ was arbitrary the proposition follows.
\end{proof}

\section{Filtrations of the ring of sections} \label{filt}

First we recall what is meant by a filtration of a graded algebra.

\begin{definition}
  A \emph{filtration} $\mathcal{F}$ of a graded algebra $\oplus_k V_k$
  is a vector space-valued map from $\mathbb{R}\times \mathbb{N},$
  $$\mathcal{F}: (t,k) \longmapsto \mathcal{F}_t V_k,$$ such that for
  any $k,$ $\mathcal{F}_t V_k$ is a family of subspaces of $V_k$ that
  is decreasing and left-continuous in $t$.
\end{definition}

In \cite{Chen} Boucksom-Chen consider certain filtrations which
behaves well with respect to the multiplicative structure of the
algebra.  They give the following definition.

\begin{definition} \label{def} Let $\mathcal{F}$ be a filtration of a
  graded algebra $\oplus_k V_k$. We shall say that
  \begin{enumerate}
  \item[(i)] $\mathcal{F}$ is \emph{multiplicative} if
    $$(\mathcal{F}_t V_k)(\mathcal{F}_s V_m)\subseteq
    \mathcal{F}_{t+s}V_{k+m}$$ for all $k,m \in \mathbb{N}$ and $s,t
    \in \mathbb{R}$.
  \item[(ii)] $\mathcal{F}$ is \emph{(linearly) bounded} if there
    exists a constant $C$ such that $\mathcal{F}_{-kC} V_k=V_k$ and
    $\mathcal{F}_{kC}V_k=\{0\}$ for all $k$.
  \end{enumerate}
\end{definition}

The goal in this section is to associate an ``analytic test configuration'' $\phi_{\lambda}^{\mathcal F}$ to any bounded multiplicative filtration of the section ring $R(L) = \oplus_k H^0(kL),$ which will be used in the next section to construct an associated geodesic.

\begin{example}
  An important example for our purpose will be the filtrations constructed from an algebraic test configuration (Section \ref{filt2}).  Another kind of example to have in mind come when $(X,L)$ is toric with moment polytope $\Delta$.  Let $f\colon \Delta\to \mathbb R$ be a bounded positive concave function.  Recalling that $\Delta\cap k^{-1}\mathbb Z^n$ parametrizes a toric basis $\{s_{\alpha}^{(k)}\}$ for $H^0(kL)$ define
$$\mathcal F_t H^0(kL) = \operatorname{linspan}\{ s_\alpha^{(k)} : f(\alpha)\ge k^{-1}t\}\subset H^0(kL).$$
It is easy to check that the concavity of $f$ implies that $\mathcal F$ is multiplicative and that $f$ is bounded implies that $\mathcal F$ is linearly bounded.  When $f$ is rational piecewise linear it turns out that this is precisely the filtration associated to the algebraic test configuration defined by $f$ as considered by Donaldson \cite{Donaldson2} (see also \cite{Sz}).   In this way one sees how the analytic test configurations and associated geodesics considered in this section and the next generalize the algebraic picture.
\end{example}

Now fix $\phi\in \mathcal{H}(L),$ and let $dV$ be a smooth volume form
on $X$ with unit mass. This gives the $L^2$-scalar product on
$H^0(kL)$ by letting $$(s,t)_{k\phi}:=\int_X
s(z)\overline{t(z)}e^{-k\phi(z)}dV(z).$$ For any $\lambda\in
\mathbb{R}$ let $\{s_{i,\lambda}\}$ be an orthonormal basis for
$\mathcal{F}_{k\lambda}H^0(kL)$ and define
$$\phi_{k,\lambda}:=\frac{1}{k}\ln(\sum |s_{i,\lambda}|^2),$$ which is
a positive metric on $L$.

\begin{lemma} \label{lemmafilt1} For any $\lambda$, the sequence of
  metrics $\phi_{k,\lambda}$ converges to a limit as $k$ tends to
  infinity, and the upper semicontinuous regularization of the limit $$\phi_{\lambda}^{\mathcal F}:={\lim}^*_{k\to \infty}  \phi_{k,\lambda}$$ is a positive metric. 
\end{lemma}

\begin{proof}
  Since $$K_{\lambda}(z,w):=\sum_i
  s_{i,\lambda}(z)\overline{s_{i,\lambda}(w)}$$ is a reproducing
  kernel of $\mathcal{F}_{k\lambda}H^0(kL)$ with respect to
  $(\cdot,\cdot)_{k\phi},$ as for the full Bergman kernel we have the
  following useful characterization (see, for example, \cite[(4.3)]{Berman})
  \begin{equation} \label{equation100} \sum
    |s_{i,\lambda}|^2=\sup\{|s|^2: s\in \mathcal{F}_{k\lambda}H^0(kL),
    \|s\|_{k\phi}^2\leq 1\}.
  \end{equation}
  Let $\|s\|_{\infty}^2:=\sup_{z\in X}\{|s(z)|^2e^{-k\phi}\}$ and
  define $$F_{k,\lambda}(z):=\sup\{|s(z)|^2: s\in
  \mathcal{F}_{k\lambda}H^0(kL), \|s\|_{\infty}^2\leq 1\}.$$ We trivially have the upper bound  $$F_{k,\lambda}(z)\leq e^{-k\phi(z)}.$$ It follows that $$\usc (\frac{1}{k}\ln F_{k,\lambda})={\sup}^*\{\frac{1}{k}\ln |s|^2: s\in
  \mathcal{F}_{k\lambda}H^0(kL), \|s\|_{\infty}^2\leq 1\})$$ is a positive metric. Let
  $\lambda$ be fixed, pick a point $z\in X,$ and let for all $k,$
  $s_k\in \mathcal{F}_{k\lambda}H^0(kL)$ be such that
  $\|s_k\|_{\infty}=1$ and $$F_{k,\lambda}(z)=|s_k(z)|^2.$$ Since the
  product $s_ks_m$ lies in $\mathcal{F}_{(k+m)\lambda}H^0((k+m)L)$ by
  the multiplicativity of $\mathcal{F},$ and $\|s_ks_m\|_{\infty}\leq
  \|s_k\|_{\infty}\|s_m\|_{\infty},$ we get that
  \begin{equation} \label{submult} F_{k+m,\lambda}(z)\geq
    F_{k,\lambda}(z)F_{m,\lambda}(z),
  \end{equation} 
  so the map $k\mapsto F_{k,\lambda}(z)$ is multiplicative. The
  existence of a limit $$\lim_{k\to \infty}\frac{1}{k}\ln
  F_{k,\lambda}(z)$$ thus follows from Fekete's lemma (see e.g.\
  \cite[p37]{Bollobas}). Since we assumed that $dV$ had unit mass, for any section $s$ $$\|s\|^ 2_{k\phi}\leq \|s\|^2_{\infty},$$ and
  thus by equation $(\ref{equation100})$ $$\sum |s_{i,\lambda}(z)|^2\geq
  F_{k,\lambda}(z).$$ On the other hand, by the Bernstein-Markov
  property of any volume form $dV$ we have that for any $\epsilon>0$
  there exists a constant $C_{\epsilon}$ so that
  $$\|s\|^2_{\infty}\leq C_{\epsilon}e^{\epsilon k}\|s\|^2_{k\phi},$$
  and thus
  \begin{equation} \label{bern} \sum |s_{i,\lambda}(z)|^2\leq
    C_{\epsilon}e^{\epsilon k}F_{k,\lambda}(z).
  \end{equation}
  It follows that the difference
  $\phi_{k,\lambda}(z)-\frac{1}{k}\ln F_{k,\lambda}(z)$ tends to zero
  as $k$ tends to infinity, thus the convergence of $\phi_{k,\lambda}$ follows. 
  
  By the multiplicativity, for any $k\in \mathbb{N}$ $$\frac{1}{k}\ln F_{k,\lambda}\leq \lim _{l\to \infty}\frac{1}{l}\ln F_{l,\lambda}=\lim _{l\to \infty}\phi_{l,\lambda},$$ and thus 
  \begin{equation} \label{sdfg}
  \usc (\frac{1}{k}\ln F_{k,\lambda})\leq {\lim}^*_{l\to \infty}\phi_{l,\lambda}=:\phi_{\lambda}^{\mathcal{F}}.
  \end{equation}
  On the other hand, clearly 
$$\lim _{l\to \infty}\phi_{l,\lambda}\leq {\sup}_k \{\usc (\frac{1}{k}\ln F_{k,\lambda})\},$$ and it follows that $$\phi_{\lambda}^{\mathcal{F}}={\sup}^*_k \{\usc (\frac{1}{k}\ln F_{k,\lambda})\}$$ so $\phi_{\lambda}^{\mathcal{F}}$ is indeed a positive metric. 
\end{proof}
\begin{remark}
  Since all volume forms $dV$ on $X$ are equivalent, the limit
  $\phi_{\lambda}$ does not depend on the choice of volume form $dV$.
\end{remark}

\begin{lemma} \label{convrew}
We have that $$\phi_{k,\lambda}\leq \phi_{\lambda}^{\mathcal{F}}+\epsilon(k),$$ where $\epsilon(k)$ is a constant independent of $\lambda$ that tends to zero as $k$ tends to infinity.
\end{lemma}

\begin{proof}
By combining the inequalities (\ref{bern}) and (\ref{sdfg}) from the proof of the the previous lemma we see that for any $\epsilon>0$ there exists a constant $C_{\epsilon}$ independent of $\lambda$ such that $$ \phi_{k,\lambda}\leq \phi_{\lambda}^{\mathcal F}+\epsilon+(1/k)\ln C_{\epsilon}.$$ This yields the lemma.
\end{proof}

\begin{proposition}
  The map $\lambda\mapsto \phi_{\lambda}^{\mathcal F}$ is a test
  curve.
\end{proposition}

\begin{proof}
  Let $\lambda$ be such that $F_{k\lambda}H^0(kL)=H^0(kL)$ for all
  $k$. Then $\phi_{k,\lambda}$ is the usual Bergman metric, and by the
  result on Bergman kernel asymptotics due to Bouche-Catlin-Tian-Zelditch (see Section \ref{sectionprelim}) we
  get that $\phi_{k,\lambda}$ converges to $\phi$. Trivially we see that if
  $F_{k\lambda}H^0(kL) = \{0\}$ for all $k$ then
  $\phi_{\lambda}^{\mathcal F}\equiv -\infty$.  By the boundedness of
  the filtration we thus have $\phi_{\lambda}^{\mathcal F}=\phi$ for
  $\lambda<-C$ and $\phi_{\lambda}^{\mathcal F}\equiv -\infty$ for
  $\lambda>C$.

  By the multiplicativity of the filtration
  $\phi_{\lambda}\equiv -\infty$ if and only if for all $k,$
  $$\mathcal{F}_{k\lambda}H^0(kL)=\{0\}.$$ Pick a $\lambda$ such that
  $\phi_{\lambda}^{\mathcal F}\not \equiv -\infty,$ then for some $k,$
  $\mathcal{F}_{k\lambda}H^0(kL)$ is non-trivial. From Lemma \ref{convrew} it follows that $\phi_{\lambda}^{\mathcal F}$ has small unbounded locus since $\phi_{k,\lambda}$ has small unbounded locus.

  It remains to prove concavity.  Let $\lambda_1,\lambda_2\in
  \mathbb{R}$ and let $t$ be a rational point in the unit interval.
  Let $m$ be a natural number such that $mt$ is an integer. Given a
  point $z\in X,$ let $s_1\in \mathcal{F}_{k\lambda_1}H^0(kL)$ and
  $s_2\in \mathcal{F}_{k\lambda_2}H^0(kL)$ be two sections with
  $\|s_1\|_{\infty}=\|s_2\|_{\infty}=1$ such that
  $$F_{k,\lambda_1}=|s_1(z)|^2$$ and $$F_{k,\lambda_2}=|s_2(z)|^2.$$
  By the multiplicativity of the filtration,
  $$s_1^{mt}s_2^{m(1-t)}\in
  F_{mk(t\lambda_1+(1-t)\lambda_2)}H^0(mkL),$$ and trivially
  $\|s_1^{mt}s_2^{m(1-t)}\|_{\infty}\leq 1$. It follows that
  $$F_{mk,t\lambda_1+(1-t)\lambda_2}(z)\geq
  F_{k,\lambda_1}(z)^{mt}F_{k,\lambda_2}(z)^{m(1-t)}.$$ Taking the
  logarithm on both sides, dividing by $mk,$ and taking the limit
  yields
  \begin{equation} \label{convconv}
    \phi_{t\lambda_1+(1-t)\lambda_2}^{\mathcal F}\geq
    t\phi_{\lambda_1}^{\mathcal F}+(1-t)\phi_{\lambda_2}^{\mathcal F}
  \end{equation}
except possibly on the pluripolar set where the limits are not equal to their upper semicontinuous regularization.  

Now if two plurisubharmonic functions $\zeta_1$ and $\zeta_2$ are equal almost everywhere then they are equal everywhere \cite[2.9.8]{Klimek}.  Applying this to $\zeta_1$ and $\max\{\zeta_1,\zeta_2\}$ we see that if $\zeta_1\ge \zeta_2$ almost everywhere then this is true everywhere.  Thus in fact  (\ref{convconv}) holds on the whole of $X$. 

  Recall that $t$ was assumed to be rational. If $\lambda_1\leq
  \lambda_2,$ the left-hand side of (\ref{convconv}) is decreasing in
  $t$ since clearly $\phi_{\lambda}^{\mathcal F}$ is decreasing in
  $\lambda$. The right-hand side of (\ref{convconv}) is continuous in
  $t,$ so it follows that the equation (\ref{convconv}) holds for all
  $t\in (0,1),$ i.e.\ $\phi_{\lambda}^{\mathcal F}$ is concave in
  $\lambda$.
\end{proof}

\begin{lemma} \label{lemma100} For any two $\phi,\psi\in
  \mathcal{H}(L)$ and any $\lambda\in \mathbb{R}$ we have
  $\phi_{\lambda}^{\mathcal F}\sim \psi_{\lambda}^{\mathcal F}$.
\end{lemma}

\begin{proof}
  If $\phi\leq \psi$ then for all  $k$
  and $\lambda$ we have that $\phi_{k,\lambda}\leq \psi_{k,\lambda},$
  and so $\phi_{\lambda}^{\mathcal F}\leq
  \psi_{\lambda}^{\mathcal F}$. Also it is clear that
  $(\phi+C)_{k,\lambda}=\phi_{k,\lambda}+C,$ which proves the lemma.
\end{proof}

\begin{definition}
  We call the map $\lambda \mapsto [\phi_{\lambda}^{\mathcal F}]$ the
  \emph{analytic test configuration associated to the filtration}
  $\mathcal F$.
\end{definition}

So by the previous lemma this analytic test configuration depends only
on $\mathcal F$ and not on the choice of $\phi\in \mathcal H(L)$.  Our
next goal is to show the curve $\phi_{\lambda}^{\mathcal F}$ is
maximal for $\lambda<\lambda_c$, for which we will need a Skoda-type division theorem.

\begin{theorem} \label{skoda} Let $L$ be an ample line bundle. Assume
  that $L$ has a smooth positive metric $\phi$ with the property that
  $dd^c\phi \geq dd^c \phi_{K_X}$ for some smooth metric $\phi_{K_X}$ on the
  canonical bundle $K_X$. Let $\{s_i\}$ be a finite collection of
  holomorphic sections of $L$ and $m> n+2$ where $n=\dim X$.

Suppose  $s$ is a section of $mL$ such
  that $$\int_X \frac{|s|^2}{(\sum |s_i|^2)^m}dV<\infty.$$ Then there exists
  sections $h_{\alpha}\in H^0((n+1)L)$ such that
  $$s=\sum_{\alpha}h_{\alpha}s^{\alpha},$$ where $\alpha$ is a multiindex $\alpha=(\alpha_i)$ with $\sum_i \alpha_i=m-n-1, $ and $s^{\alpha}$ are the monomials $s^{\alpha}:=\Pi_i s_i^{\alpha_i}$.
\end{theorem}

\begin{proof}
Let $k$ be an integer such that $n+2\le k\le m$.  Then given a section $t\in H^0(kL)$ with $$\int_X \frac{|t|^2}{(\sum_i |s_i|^2)^k} dV<\infty$$  an application of the Skoda division theorem \cite[Thm. 2.1]{Dror} yields sections $\{t_i\}$ of $(k-1)L$ such that $ t= \sum_i t_i s_i$
and
$$\int_X \frac{|t_i|^2}{(\sum_i |s_i|^2)^{k-1}} dV<\infty.$$ 
(To apply the cited theorem replace $F,E,\psi,\eta$ with $kL -K_X, L, k\phi-\phi_{K_X}, \phi$ respectively and replace $\alpha q$ with $k-1> n+1$.)   

Now we first apply the above with $k=m$ to the section $s$, and then apply again with $k=m-1$ to each of the sections $t_i$.  Repeating this process with $k=m,m-1,\ldots,n+2$ we see that $s$ can be written as a linear sum of monomials in the $s_i$ as required.
\end{proof}

\begin{proposition} \label{lemma202}
For $\lambda$ less than the critical value $\lambda_c$,
$$\phi_{\lambda}^{\mathcal F}=\lim_{k\to \infty} \phi_{[\phi_{k,\lambda}]}.$$
\end{proposition}

\begin{proof}
  Let $\phi_k:=\phi_{k,-\infty},$ i.e.\ the Bergman metric $1/k\ln
  (\sum|s_i|^2),$ where $\{s_i\}$ is an orthonormal basis for the
  whole space $H^0(kL)$ with respect to $(\cdot,\cdot)_{k\phi}$. By
  the Bernstein-Markov property of any volume form $dV$ (see e.g.\
  \cite{Witt}), or simply the maximum principle,
  \begin{equation} \label{equation101} \phi_k\leq \phi+ \epsilon_k,
  \end{equation} where $\epsilon_k$ tends to zero as $k$ tends to infinity. Since $\phi_{k,\lambda}$ is decreasing in $\lambda,$ the inequality (\ref{equation101}) still holds when $\phi_k$ is replaced by $\phi_{k,\lambda},$ i.e.\ $\phi_{k,\lambda}-\epsilon_k\leq \phi$. Therefore $\phi_{k,\lambda}-\epsilon_k$ belongs to the class of metrics the supremum of which yields $P_{[\phi_{k,\lambda}]}\phi,$ so  $$\phi_{k,\lambda}\leq P_{[\phi_{k,\lambda}]}\phi+\epsilon_k \le \phi_{[\phi_{k,\lambda}]} + \epsilon_k,$$ so letting $k$ tend to infinity
$$\phi_{\lambda}^{\mathcal F}\leq \phi_{[\phi_{k,\lambda}]}.$$

For the other inequality it is enough to show that for any constant $C,$
\begin{equation} \label{trew}
P_{\phi_{k,\lambda}+C}\, \phi\leq \phi_{\lambda}^{\mathcal F}.
\end{equation} 
By the assumption that $\lambda<\lambda_c$ it follows that $\phi_{\lambda}^{\mathcal F}\not \equiv -\infty$.
Let $\psi$ be a positive metric dominated by both $\phi_{k,\lambda}+C$
and $\phi,$ where $k$ is large enough so that $kL$ fulfills the
requirements of Theorem \ref{skoda}. We denote by $\mathcal{J}(k\psi)$ the multiplier ideal sheaf of germs of holomorphic functions locally
integrable against $e^{-k\psi}$. Let $\{s_i\}$ be an orthonormal basis
of $H^0(kL\otimes \mathcal{J}(k\psi)),$ and denote by $\psi_k$ the
Bergman metric
$$\psi_k:=\frac{1}{k}\ln (\sum |s_i|^2).$$ 
By Theorem \ref{demailly}, $$\psi\leq
\psi_k+\delta_k$$ where $\delta_k$ tends to zero as $k$ tends to
infinity, and $\psi_k$ converges pointwise to $\psi$. If $s$ lies in
$H^0(kL\otimes \mathcal{J}(k\psi)),$ specifically we must have that
$$\int_X\frac{|s|^2}{\sum|s_{i,\lambda}|^2}dV<\infty,$$ since we
assumed that $\psi$ was dominated by $\phi_{k,\lambda}+C=1/k\ln
(\sum|s_{i,\lambda}|^2)+C$. Similarly if $s$ lies in $H^0(kmL\otimes
\mathcal{J}(km\psi))$ we have
$$\int_X\frac{|s|^2}{(\sum|s_{i,\lambda}|^2)^m}dV<\infty.$$ From
Theorem \ref{skoda} applied to the sections $\{s_{i,\lambda}\}$ it
thus follows that $$s=\sum h_{\alpha}s^{\alpha},$$ where
$h_{\alpha}\in H^0(k(n+1)L),$ and the $s^{\alpha}$ are monomials in the
$\{s_{i,\lambda}\}$ of degree $m-n-1$.  Because of the
multiplicativity of the filtration each $s^{\alpha}$ lies in
$$\mathcal{F}_{k(m-n-1)\lambda} H^0(k(m-n-1)L),$$ and by the boundedness of
the filtration we also have that each $h_{\alpha}$ lies in
$$F_{-k(n+1)C}H^0(k(n+1)L)$$ for some fixed constant $C$. We thus get that $  H^0(kmL\otimes \mathcal{J}(km\psi))$ is contained in
\begin{eqnarray} \label{cfg}
(F_{-k(n+1)C}H^0(k(n+1)L))(\mathcal{F}_{k(m-n-1)\lambda}H^0(k(m-n-1)L))\nonumber\\ \subseteq \mathcal{F}_{k(m-n-1)\lambda-k(n+1)C}H^0(kmL).
\end{eqnarray}
Since we assumed that $\psi\leq \phi$ we have that $\psi_{km}$ is less than or equal to the Bergman metric using an orthonormal basis for $H^0(kmL\otimes \mathcal{J}(km\psi))$ with respect to $\phi$. Because of (\ref{cfg}) this Bergman metric is certainly less than or equal to $\phi_{km, \lambda'},$ where
$$\lambda':=\frac{1}{km}(k(m-n-1)\lambda-k(n+1)C).$$ Hence $$\psi_{km}\leq \phi_{km, \lambda'}.$$ On the other hand, by Lemma \ref{convrew} we get $$\phi_{km, \lambda'} \leq \phi_{\lambda'}^{\mathcal{F}}+\epsilon(km),$$ where $\epsilon(km)$ is a constant independent of $\lambda'$ that tends to zero as $km$ tends to infinity. Since $\lambda'$ tends
to $\lambda$ as $m$ tends to infinity this implies $\psi\leq
\lim_{\lambda'\to \lambda}\phi_{\lambda'}^{\mathcal F},$ and thus by Lemma \ref{leftcontlemma} $\psi\leq
\phi_{\lambda}^{\mathcal F}$. Taking the supremum over all such
$\psi$ completes the proof.
\end{proof}

\begin{corollary} \label{lemma203} Suppose $\mathcal F$ is a multiplicative linearly bounded filtration of $\oplus_k H^0(kL)$.  Then the associated test curve $\phi_{\lambda}^{\mathcal F}$ is maximal for $\lambda<\lambda_c$ and its Legendre transform is a geodesic ray.
\end{corollary}

\begin{proof}
  Theorem \ref{lemma3} tells us that $\phi_{[\phi_{k,\lambda}]}$ is
  maximal with respect to $\phi=\phi_{-\infty}$. By Lemma
  \ref{lemlim} it follows that this is true for the limit
  $\phi_{\lambda}^{\mathcal F} = \lim_{k\to \infty}
  \phi_{[\phi_{k,\lambda}]}$ as well. Let $\phi_{\lambda}$ be the test curve defined by $\phi_{\lambda}:=\phi_{\lambda}^{\mathcal F}$ for $\lambda<\lambda_c$ and $\phi_{\lambda}\equiv -\infty$ for $\lambda\geq \lambda_c$. Thus $\phi_{\lambda}$ is a maximal test curve, thus its Legendre transform is a geodesic ray. On the other hand, for every $\epsilon>0$ $$\phi_{\lambda}\leq\phi_{\lambda}^{\mathcal F}\leq \phi_{\lambda-\epsilon},$$ and therefore $$\widehat{\phi}_t\leq \widehat{(\phi^{\mathcal{F}})}_t\leq \widehat{\phi}_t+\epsilon t.$$ Since $\epsilon$ was arbitrary we get that the Legendre transform of $\phi_{\lambda}^{\mathcal{F}}$ coincides with that of $\phi_{\lambda},$ and thus it is a geodesic ray.
\end{proof}

\begin{remark}
  Given an analytic test configuration $[\psi_{\lambda}]$ there is a
  naturally associated filtration $\mathcal{F}$ of the section ring,
  defined as $$\mathcal{F}_{k\lambda}H^0(kL):=H^0(kL\otimes
  \mathcal{J}(k\psi_{\lambda})).$$ This filtration is bounded, but in
  general not multiplicative.
\end{remark}

\section{Filtrations associated to algebraic test configurations}
\label{filt2}

We recall briefly Donaldson's definition of a test configuration
\cite{Donaldson2,Donaldson}. In order to not confuse them with the our
analytic test configurations, we will in this article refer to them as
algebraic test configurations.

\begin{definition}
  An \emph{algebraic test configuration} $\mathcal{T}$ for an ample
  line bundle $L$ over $X$ consists of:
  \begin{enumerate}
  \item[(i)] a scheme $\mathcal{X}$ with a
    $\mathbb{C}^{\times}$-action $\rho,$
  \item[(ii)] a $\mathbb{C}^{\times}$-equivariant line bundle
    $\mathcal{L}$ over $\mathcal{X},$
  \item[(iii)] and a flat $\mathbb{C}^{\times}$-equivariant projection
    $\pi: \mathcal{X} \to \mathbb{C}$ where $\mathbb{C}^{\times}$ acts
    on $\mathbb{C}$ by multiplication, such that $\mathcal{L}$ is
    relatively ample, and such that if we denote by
    $X_1:=\pi^{-1}(1)$, then $\mathcal{L}_{|X_1} \to X_1$ is
    isomorphic to $rL \to X$ for some $r>0$.
  \end{enumerate}
\end{definition}

By rescaling we can without loss of generality assume that $r=1$ in
the definition.  An algebraic test configuration is called a \emph{product test configuration} if there is a $\mathbb{C}^{\times}$-action $\rho'$ on $L \to X$ such
that $\mathcal{L}=L\times \mathbb{C}$ with $\rho$ acting on $L$ by
$\rho'$ and on $\mathbb{C}$ by multiplication. An algebraic test
configuration is called \emph{trivial} if it is a product test configuration
with the action $\rho'$ being the trivial
$\mathbb{C}^{\times}$-action.

Since the zero-fiber $X_0:=\pi^{-1}(0)$ is invariant under the action
$\rho$, we get an induced action on the space $H^0(kL_0),$ also
denoted by $\rho,$ where we have denoted the restriction of
$\mathcal{L}$ to $X_0$ by $L_0$. Specifically, we let $\rho(\tau)$ act
on a section $s\in H^0(kL_0)$ by
\begin{equation} \label{eq21}
  (\rho(\tau)(s))(x):=\rho(\tau)(s(\rho^{-1}(\tau)(x))).
\end{equation}
By standard theory any vector space $V$ with a
$\mathbb{C}^{\times}$-action can be split into weight spaces
$V_{\lambda_i}$ on which $\rho(\tau)$ acts as multiplication by
$\tau^{\lambda_i}$, (see e.g.\ \cite{Donaldson2}). The numbers
$\lambda_i$ with non-trivial weight spaces are called the weights of
the action. Thus we may write $H^0(kL_0)$ as
$$H^0(kL_0)=\oplus_{\lambda}V_{\lambda}$$ with respect to the induced
action $\rho$.

In \cite[Lem. 4]{Sturm} Phong-Sturm give the following linear bound
on the absolute value of the weights.

\begin{lemma} \label{phong} Given a test configuration there is a
  constant $C$ such that $$|\lambda_i|<Ck$$ whenever $\dim
  V_{\lambda_i}>0$.
\end{lemma}

In \cite{Witt2} the second author showed how to get an associated
filtration $\mathcal{F}$ of the section ring $\oplus_k H^0(kL)$ given
a test configuration $\mathcal{T}$ of $L$ which we now recall.

First note that the $\mathbb{C}^{\times}$-action $\rho$ on
$\mathcal{L}$ via the equation (\ref{eq21}) gives rise to an induced
action on $H^0(\mathcal{X}, k\mathcal{L})$ as well as
$H^0(\mathcal{X}\setminus X_0,k\mathcal{L}),$ since
$\mathcal{X}\setminus X_0$ is invariant.  Let $s\in H^0(kL)$ be a holomorphic section. Then using the
$\mathbb{C}^{\times}$-action $\rho$ we get a canonical extension
$\bar{s}\in H^0(\mathcal{X}\setminus X_0,k\mathcal{L})$ which is
invariant under the action $\rho$, simply by letting
\begin{equation} \label{toroto} \bar{s}(\rho(\tau)x):=\rho(\tau)s(x)
\end{equation} 
for any $\tau\in \mathbb{C}^{\times}$ and $x\in X$.

We identify the coordinate $z$ with the projection function $\pi(x),$
and we also consider it as a section of the trivial bundle over
$\mathcal{X}$. Exactly as for $H^0(\mathcal{X}, k\mathcal{L}),$ $\rho$
gives rise to an induced action on sections of the trivial bundle,
using the same formula (\ref{eq21}).  From this one sees
\begin{equation} \label{useful}
  (\rho(\tau)z)(x)=\rho(\tau)(z(\rho^{-1}(\tau)x)=\rho(\tau)(\tau^{-1}z(x))=\tau^{-1}z(x),
\end{equation}
where we used that $\rho$ acts on the trivial bundle by multiplication
on the $z$-coordinate. Thus $$\rho(\tau)z=\tau^{-1}z,$$ which shows
that the section $z$ has weight $-1$.

By this it follows that for any section $s\in H^0(kL)$ and any integer
$\lambda,$ we get a section $z^{-\lambda}\bar{s}\in
H^0(\mathcal{X}\setminus X_0,k\mathcal{L}),$ which has weight
$\lambda$.

\begin{lemma} \label{igenlemma} For any section $s\in H^0(kL)$ and any
  integer $\lambda$ the section $z^{-\lambda}\bar{s}$ extends to a
  meromorphic section of $k\mathcal{L}$ over the whole of
  $\mathcal{X},$ which we also will denote by $z^{-\lambda}\bar{s}$.
\end{lemma}

\begin{proof}
  It is equivalent to saying that for any section $s$ there exists an
  integer $\lambda$ such that $z^{\lambda}\bar{s}$ extends to a
  holomorphic section $S\in H^0(\mathcal{X},k\mathcal{L})$. By
  flatness, which was assumed in the definition of a test
  configuration, the direct image bundle $\pi_*\mathcal{L}$ is in fact
  a vector bundle over $\mathbb{C}$. Thus it is trivial, since any
  vector bundle over $\mathbb{C}$ is trivial. Therefore there exists a
  global section $S'\in H^0(\mathcal{X},k\mathcal{L})$ such that
  $s=S'_{|X}$. On the other hand, as for $H^0(kL_0),$
  $H^0(\mathcal{X},k\mathcal{L})$ may be decomposed as a direct sum of
  invariant subspaces $W_{\lambda'}$ such that $\rho(\tau)$ restricted
  to $W_{\lambda'}$ acts as multiplication by $\tau^{\lambda'}$. Let
  us write
  \begin{equation} \label{decomp} S'=\sum S'_{\lambda'},
  \end{equation}
  where $S_{\lambda'}\in W_{\lambda'}$. Restricting the equation
  (\ref{decomp}) to $X$ gives a decomposition of $s,$ $$s=\sum
  s_{\lambda'},$$ where $s_{\lambda'}:={S'_{\lambda'}}_{|X}$. From
  (\ref{toroto}) and the fact that $S'_{\lambda'}$ lies in
  $W_{\lambda'}$ we get that for $x\in X$ and $\tau\in
  \mathbb{C}^{\times}$
  \begin{eqnarray*}
    \bar{s}_{\lambda'}(\rho(\tau)(x))=\rho(\tau)(s_{\lambda'}(x))=\rho(\tau)(S'_{\lambda'}(x))=(\rho(\tau)S'_{\lambda'})(\rho(\tau)(x)))=\\ =\tau^{\lambda'}S'_{\lambda'}(\rho(\tau)(x)),
  \end{eqnarray*} 
  and therefore $\bar{s}_{\lambda'}=\tau^{\lambda'}S'_{\lambda'}$.
  Since trivially $$\bar{s}=\sum \bar{s}_{\lambda'}$$ it follows that
  $t^{\lambda}\bar{s}$ extends holomorphically as long as $\lambda\geq
  \max -\lambda'$.
\end{proof}

\begin{definition}
  Given a test configuration $\mathcal{T}$ we define a vector
  space-valued map $\mathcal{F}$ from $\mathbb{Z}\times \mathbb{N}$ by
  letting $$(\lambda,k) \longmapsto \{s\in H^0(kL):
  z^{-\lambda}\bar{s}\in
  H^0(\mathcal{X},k\mathcal{L})\}=:\mathcal{F}_{\lambda}H^0(kL).$$
\end{definition}
It is immediate that $\mathcal{F}_{\lambda}$ is decreasing since
$H^0(\mathcal{X},k\mathcal{L})$ is a $\mathbb{C}[z]$-module. We can
extend $\mathcal{F}$ to a filtration by letting
$$\mathcal{F}_{\lambda}H^0(kL):=\mathcal{F}_{\lceil\lambda\rceil}H^0(kL)$$
for non-integers $\lambda,$ thus making $\mathcal{F}$ left-continuous.
Since
$$z^{-(\lambda+\lambda')}\overline{ss'}=(z^{-\lambda}\bar{s})(z^{-\lambda'}\bar{s'})\in
H^0(\mathcal{X},k\mathcal{L})H^0(\mathcal{X},m\mathcal{L})\subseteq
H^0(\mathcal{X},(k+m)\mathcal{L})$$ whenever $s\in
\mathcal{F}_{\lambda}H^0(kL)$ and $s'\in
\mathcal{F}_{\lambda'}H^0(kL),$ we see that $$(\mathcal{F}_{\lambda}
H^0(kL))(\mathcal{F}_{\lambda'} H^0(mL))\subseteq
\mathcal{F}_{\lambda+\lambda'}H^0((k+m)L),$$ i.e.\ $\mathcal{F}$ is
multiplicative.

Recall that we had the decomposition of $H^0(kL_0)$ into weight spaces
$V_{\lambda}$.

\begin{lemma} \label{lemma200} For each $\lambda,$ we have that $$\dim
  F_{\lambda}H^0(kL)=\sum_{\lambda'\geq \lambda} \dim V_{\lambda'}.$$
\end{lemma}

\begin{proof}
  We have the following isomorphism:
  \begin{equation*}
    (\pi_*k\mathcal{L})_{|\{0\}}\cong H^0(\mathcal{X},k\mathcal{L})/zH^0(\mathcal{X},k\mathcal{L}),
  \end{equation*}
  the right-to-left arrow being given by the restriction map (see e.g.\
  \cite[p12]{Ross2}). Also, for $k\gg 0$,
  $(\pi_*k\mathcal{L})_{|\{0\}}=H^0(kL_0),$ therefore for
  large $k$
  \begin{equation} \label{isohej} H^0(kL_0)\cong
    H^0(\mathcal{X},k\mathcal{L})/zH^0(\mathcal{X},k\mathcal{L}),
  \end{equation}
  We also had a decomposition of $H^0(\mathcal{X},k\mathcal{L})$ into
  the sum of its invariant weight spaces $W_{\lambda}$. By Lemma
  \ref{igenlemma} it is clear that a section $S\in
  H^0(\mathcal{X},k\mathcal{L})$ lies in $W_{\lambda}$ if and only if
  it can be written as $z^{-\lambda}\bar{s}$ for some $s\in H^0(kL),$
  in fact we have that $s=S_{|X}$. Thus $$W_{\lambda}\cong
  \mathcal{F}_{\lambda}H^0(kL),$$ and by the isomorphism
  (\ref{isohej}) then $$V_{\lambda}\cong
  \mathcal{F}_{\lambda}H^0(kL)/\mathcal{F}_{\lambda+1}H^0(kL).$$ Therefore
  \begin{equation} \label{slutnej} \dim
    \mathcal{F}_{\lambda}H^0(kL)=\sum_{\lambda'\geq \lambda}\dim
    V_{\lambda'}.
  \end{equation}
\end{proof}

Using Lemma \ref{lemma200} together with Lemma \ref{phong} shows that
the filtration $\mathcal{F}$ is bounded.

\section{The geodesic rays of Phong and
  Sturm}\label{sectionphongsturm}

In \cite{Sturm} Phong-Sturm show how to construct a weak geodesic ray, starting with a $\phi\in \mathcal{H}(L)$ and an algebraic test configuration $\mathcal{T}$ (see also \cite{Zelditch} for how this works in the toric setting).  In the previous section we showed how to associate an analytic test configuration $[\phi^{\mathcal{F}}_{\lambda}]$ to an algebraic test configuration, and thus get a weak geodesic using the Legendre transform of its maximal envelope.  Recall by Proposition \ref{vasd} this geodesic is the same as the Legendre transform of the original test curve $\phi^{\mathcal{F}}_{\lambda}$. The goal in this section is to prove that this ray coincides with the one constructed by Phong-Sturm.\medskip

To describe what we aim to show, recall that if $V$ is a vector space with a scalar product, and $\mathcal{F}$ is a
filtration of $V,$ there is a unique decomposition of $V$ into a
direct sum of mutually orthogonal subspaces $V_{\lambda_i}$ such that
$$\mathcal{F}_{\lambda}V=\oplus_{\lambda_i\geq
  \lambda}V_{\lambda_i}.$$ Furthermore we allow for
$\lambda_i$ to be equal to $\lambda_j$ even when $i\neq j,$ so we
can assume that all the subspaces $V_{\lambda_i}$ are one dimensional. 
This additional decomposition is of course not unique, but it will not
matter in what follows.

Let $\phi\in \mathcal{H}(L)$ and $H^0(kL)=\oplus V_{\lambda_i}$ be
the decomposition of $H^0(kL)$ with respect to the scalar product
$(\cdot,\cdot)_{k\phi}$ coming from the volume form $(dd^c\phi)^n$.  Consider next the filtration coming from an algebraic
test configuration  (note that then the collection of $\lambda_i$ will depend also on $k$ but we omit that from our notation) and define the normalized weights to be
$$ \bar{\lambda}_i := \frac{\lambda_i}{k},$$
which form a bounded family by Lemma \ref{phong}.

 Now if  $s_i$ is a vector of unit length in $V_{\lambda_i},$ then $\{s_i\}$ will be an orthonormal basis for $H^0(kL)$. Since the filtration $\mathcal{F}$ encodes the $\mathbb{C}^*$-action on $H^0(k\mathcal{L})$ it is easy to see that the basis $\{s_i\}$ is the same one as in \cite[Lem 7]{Sturm}.     In terms of the notation in the previous sections
$$\phi_{k,\lambda} = \frac{1}{k}\ln (\sum_{\lambda_i\ge k\lambda} |s_i|^2) \quad \text{ and } \quad \phi^{\mathcal F}_{\lambda} = {\lim}^*_{k\to \infty} \phi_{k,\lambda}.$$

\begin{definition}
Let
$$\Phi_k(t):=\frac{1}{k}\ln(\sum_i e^{t\lambda_i}|s_i|^2)$$
The \emph{Phong-Sturm} ray is the limit
\begin{equation}
\Phi(t):= {\lim}^*_{k\to\infty}(\sup_{l\geq k}\Phi_l(t)).\label{equationps}
\end{equation}

\end{definition}

Our goal is the following:
\begin{theorem}\label{theoremps}
  Let $\phi^{\mathcal F}$ be the analytic test configuration associated to the filtration $\mathcal F$ from a test configuration.   Then 
$$\Phi(t) =\widehat{(\phi^{\mathcal F})}_t.$$
\end{theorem}

In particular, the results from the previous section yield another proof of \cite[Thm 1]{Sturm} which says that $\Phi(t)$ is a weak geodesic ray emanating from $\phi$.

\begin{lemma} \label{lemmafineq}
  \begin{equation}
    \Phi(t)={\lim}^*_{k\to \infty}(\sup_{l\ge k}\Phi_l(t))={\lim}^*_{k\to\infty} (\sup_{l\ge k} \max_i\{\phi_{l,\bar{\lambda}_i}+t \bar{\lambda_i}\}).
  \end{equation} 
\end{lemma}

\begin{proof}
Our proof will be based on the elementary fact that if $\{a_{l,i} : i\in I_l\}$ is a set of real numbers then
\begin{equation}\label{eq500} 
 \max_{i\in I_l} a_{l,i} \le \frac{1}{l} \ln  \sum_{i\in I_l} e^{l a_{l.i}} \le  \max_{i\in I_l} a_{l,i}  + \frac{1}{l} \ln |I_l|.
\end{equation}
Now pick $x\in X$ and $t> 0$. Let
  $$a_{l,i}:=\frac{1}{l} \ln|s_i(x)|^2+ t\bar{\lambda}_i$$
and $I_l$ be the indexing set for the $\lambda_i$.  Then $|I_l| = O(l^n)$ and
$$\Phi_l(t) =  \frac{1}{l} \ln \left(\sum_i e^{l a_{l,i}}\right).$$
Thus by \eqref{eq500}
\begin{equation}
  \label{eq:fineq1}
\max_i\{a_{l,i}\} \le \Phi_l(t) \le \max_i\{a_{l,i}\} + \frac{|I_l|}{l}.
\end{equation}
Now set
$$b_{l,i}:=\phi_{l,\bar{\lambda}_i} + t\bar{\lambda}_i= \frac{1}{l} \ln \sum_{\lambda_j\ge \lambda_i} |s_j(x)|^2+t\bar{\lambda}_i.$$
For fixed $i$, pick any $j_0$ such that
$$ \max_{\lambda_j \ge \lambda_i} |s_j(x)|^2 = |s_{j_0}|^2 \quad \text{ and } \quad \lambda_{j_0} \ge \lambda_i.$$
Then
$$b_{l,i} \le \frac{1}{l} \ln(|I_l| |s_{j_0}|^2+ t \bar{\lambda}_i  \le \frac{1}{l} \ln |s_{j_0}|^2+t\bar{\lambda}_{j_0}  + \frac{\ln |I_l|}{l} = a_{j_0,l} + \frac{\ln |I_l|}{l}.$$
Clearly $a_{l,i}\le b_{l,i}$ for all $i$, so we in fact have
$$\max_i\{a_{l,i}\}\le \max\{b_{l,i}\}\le \max_i\{a_{l,i}\}+ \frac{\ln |I_l|}{l},$$
which combined with \eqref{eq:fineq1} yields
$$\max_i\{b_{l,i}\} - \frac{\ln |I_l|}{l} \le \Phi_l(t) \le \max_i\{b_{l,i}\} + \frac{\ln |I_l|}{l}.$$
Now taking the supremum over all $l\ge k$ followed by the upper semicontinuous regularization and then the limit as $k$ tends to infinity gives the result since $k^{-1}\ln |I_k|$ tends to zero.
\end{proof}

\begin{proof}[Proof of Theorem \ref{theoremps}]
From Lemma \ref{convrew}  there is a constant $\epsilon(l)$ such that
$$ \phi_{l,\bar{\lambda}_i} + t\bar{\lambda}_i \le \phi^{\mathcal F}_{\bar{\lambda}_i} + t\bar{\lambda}_i  + \epsilon(l),$$
where $\epsilon(l)$ is independent of $\lambda_i$ and tends to zero as $l$ tends to infinity.  Thus
$$\max_i\{\phi_{l,\bar{\lambda}_i} + t\bar{\lambda}_i\} \le \sup_{\lambda} \{ \phi_{\lambda}^{\mathcal F} + t\lambda\} + \epsilon(l),$$
and so
$$ {\lim}^*_{k\to \infty} {\sup}_{l\ge k} \max_i\{\phi_{l,\bar{\lambda}_i} + t\bar{\lambda}_i\} \le (\widehat{\phi^{\mathcal F}})_t + \lim_{k\to \infty}\sup_{l\ge k} \epsilon(l),$$
so using Lemma \ref{lemmafineq} gives 
$$\Phi(t)\le (\widehat{\phi^{\mathcal F}})_t.$$

For the opposite inequality, let $\lambda\in \mathbb{R}$ be arbitrary. Trivially $$\Phi_k(t)= \frac{1}{k}\ln (\sum_i e^{t\lambda_i}|s_i|^2)\geq \frac{1}{k}\ln (\sum_{\lambda_i\geq k\lambda} e^{tk\lambda}|s_i|^2)=\phi_{k,\lambda}+t\lambda.$$ Hence  $$\Phi(t)\geq \phi_{\lambda}^{\mathcal{F}}+t\lambda$$ for any $\lambda,$ and thus $$\Phi(t)\ge (\widehat{\phi^{\mathcal F}})_t.$$ 
  
\end{proof}
 
 \begin{remark}
Phong-Sturm prove in \cite{Sturm} that the geodesic ray one gets from an algebraic test configuration $\mathcal{F}$ is non-trivial if the norm of $\mathcal{F}$ is non-zero.  From the above we see that the weak geodesic ray is trivial if and only if the associated analytic test configuration is trivial, i.e.\ if there exists a number $\lambda_c$ such that $\phi_{\lambda}=\phi$ when $\lambda<\lambda_c$ and $\phi_{\lambda}\equiv -\infty$ when $\lambda>\lambda_c.$
 \end{remark}

\noindent {\sc Julius Ross,  University of Cambridge, UK. \\j.ross@dpmms.cam.ac.uk}\vspace{2mm}\\ 
\noindent{\sc David Witt Nystr\"om,  University of Gothenburg, Sweden. \\\quad wittnyst@chalmers.se}

\end{document}